\documentclass[12pt]{amsart}
\usepackage{amsfonts}
\usepackage{amsmath, amsthm, amssymb,color}
\usepackage{latexsym}
\usepackage{txfonts}
\usepackage{bm}
\usepackage{graphicx}
\usepackage{graphics}
\usepackage{epsfig}
\usepackage[polish,english]{babel}
\usepackage[T1]{fontenc}
\usepackage{ascmac}
 
\numberwithin{equation}{section}
\allowdisplaybreaks

\def\tag{\refstepcounter{equation}\leqno }

\newtheorem{lemma}{Lemma}[section]

\newtheorem{theorem}[lemma]{Theorem}

\newtheorem{definition}[lemma]{Definition}

\newtheorem{example}[lemma]{Example}

\newtheorem{remark}[lemma]{Remark}

\def\E{{\mathbb{E}}}
\def\P{{\mathbb{P}}}
\def\R{{\mathbb{R}}}

\def\N{{\mathbb{N}}}

\def\Z{{\mathbb{Z}}}

\newcommand{\wt}{\widetilde}

\def\tag{\refstepcounter{equation}\leqno }

\newcommand{\bfPi}{\mbox{\boldmath$\Pi$}}
\newcommand{\bfx}{{\bf x}}
\newcommand{\bfX}{{\bf X}}
\newcommand{\bfB}{{\bf B}}

\newcommand{\bali}{\begin{align}}
\newcommand{\eali}{\end{align}}

\newcommand{\bfY}{{\bf Y}}
\newcommand{\bfy}{{\bf y}}
\newcommand{\bfA}{{\bf A}}

\newcommand{\bfZ}{{\bf Z}}
\newcommand{\bfW}{{\bf W}}

\newcommand{\bfS}{{\bf S}}

\begin{document}
\today
\title[Tail indices for $AX+B$ recursion with triangular matrices]{Tail indices for $\bfA\bfX+\bfB$ recursion with triangular matrices}
\author[M. Matsui]{Muneya Matsui}
\author[W. \'{S}wi\k{a}tkowski]{Witold \'{S}wi\k{a}tkowski}
\address{Department of Business Administration, Nanzan University,
18 Yamazato-cho Showa-ku Nagoya, 466-8673, Japan}
\email{mmuneya@nanzan-u.ac.jp}
\address{Institute of Mathematics
University of Wroc{\l}aw
Pl.  Grunwaldzki 2/4,
50-384, Wroc{\l}aw, Poland}
\email{witswiat@math.uni.wroc.pl}
\begin{abstract}
We study multivariate stochastic recurrence equations (SREs) with triangular matrices.
If coefficient matrices of SREs have strictly positive entries, the
 classical Kesten result says that the stationary solution is regularly
 varying and the tail indices are the same in all directions. This
 framework, however, is too restrictive for applications. In order to
 widen applicability of SREs, we consider SREs with triangular matrices and we prove that their stationary solutions are
regularly varying with component-wise different tail exponents. 
Several applications to GARCH models are suggested.   
\end{abstract}

\vspace{2mm}


\thanks{The first author is partly supported by the JSPS Grant-in-Aid
for Young Scientists B (16k16023). The second author is partly supported by the NCN Grant UMO-2014/15/B/ST1/00060}
\maketitle

\section{Introduction}

\subsection{Results and motivation}

A multivariate stochastic recurrence equation (SRE) 
\begin{equation}
 \label{sre1}
 \bfW_t=\bfA_t\bfW_{t-1}+\bfB_t,\quad t\in\N
\end{equation}
is studied, where $\bfA_t$ is a random $d\times d$ matrix with nonnegative
entries, $\bfB_t$ is a random vector in $\R^d$ with nonnegative components and 
$(\bfA_t,\bfB_t)_{t\in\Z}$ are i.i.d.
The sequence $(\bfW_t)_{t\in\N}$ generated by iterations of SRE \eqref{sre1} is a Markov chain, however, it is not necessarily stationary. Under some mild contractivity and integrability  conditions (see e.g. \cite{bougerol:picard:1992a, brandt:1986}),
$\bfW_t$ converges in distribution to a random variable $\bfW$ 
which is the unique solution to the stochastic equation
\begin{equation}
 \label{dif recurrence}
 \bfW\stackrel{d}=\bfA\bfW+\bfB.
\end{equation}
Here $(\bfA,\bfB)$ denotes a generic element of the
sequence $(\bfA_t,\bfB_t)_{t\in\Z}$, 
which is independent of $\bfW$ and the equality is meant in distribution. If we put $\bfW_0\stackrel{d}=\bfW$ then the sequence $(\bfW_t)$ of \eqref{sre1}
is stationary. 
Moreover, under suitable conditions a strictly stationary
casual solution $(\bfW_t)$ to \eqref{sre1} can be written by the formula
\[
 \bfW_t=\bfB_t+\sum_{i=-\infty}^t \bfA_t\cdots\bfA_i\bfB_{i-1}
\]
and $\bfW_t\stackrel{d}=\bfW$ for all $t\in\Z$.

Stochastic iteration \eqref{sre1} has been already studied 
for almost half a century and it has found numerous applications to financial 
models (see e.g. Section 4 of \cite{buraczewski:damek:mikosch:2016}). 
Various properties have been investigated. Our particular interest here is 
the tail behavior of the stationary
solution $\bfW$. The topic is not only interesting on its own 
but it also has applications to e.g. risk management \cite[Sec. 7.3]{mcneil:fre:embrechts:2015}.

The condition for the stationary solution $\bfW$ having power decaying
tails dates back to Kesten \cite{kesten:1973}.  
Since then, the Kesten theorem and its extensions have been used 
to characterize tails in various situations. An essential feature of Kesten-type results
is that tail behavior is the same in all coordinates. We are going to call it {\it Kesten property.}
However, this property is not necessarily shared by all interesting models -- several empirical
evidences support the fact (see e.g. \cite{sun:chen:2014,horvath:boril:2016,matsui:mikosch:2016} from
economic data). Therefore, SREs with solutions having more flexible tails are both
challenging and desirable in view of applications.

The key assumption implying Kesten property 
is an irreducibility condition and it refers to the law of $\bfA$. The simplest example when it fails, is a SRE with diagonal matrices $\bfA=diag(A_{11},\ldots,A_{dd})$. In this case the solution 
can exhibit different tail behavior in coordinates, which is not
difficult to see by the univariate Kesten result. 
In the present paper, we consider a particular
case of triangular matrices $\bfA$, which is much more complicated and much more applicable as well. We derive the precise tail asymptotics of solution in all coordinates. In particular we show that it may vary in coordinates, though it is also possible that in some coordinates we have the same tail asymptotics.

More precisely let $\bfA$ be non-negative matrices such that
$A_{ii}>0,\,A_{ij}=0,\,i> j\ a.s.$ Suppose that 
$\E A_{ii}^{\alpha_i}=1$ holds for $\alpha_i>0$ in each $i$ such that $\alpha_i\neq
\alpha_j,\,i\neq j$. 
We prove that when $x\to \infty$ 
\begin{equation}
\label{assymp}
\P (W_i >x)\sim c_i x^{-\wt\alpha_i}, \quad c_i>0, \quad
i=1,\ldots,d,  
\end{equation}
for $\wt\alpha_i>0$ depending on $\alpha_i,\ldots,\alpha_d$.
Here and in what follows the notation `$\sim$' means
that the quotient of the left and right hand sides tends to 1 as $x\to\infty$.
For $d=2$, the result \eqref{assymp} was proved in
\cite{damek:matsui:swiatkowski:2017} with indices $\wt\alpha_1= \min (\alpha _1, \alpha
_2)$ and $\wt \alpha_2=\alpha_2$.
The dependency of $\wt \alpha_1$ on $\alpha_1,\alpha_2$ comes from the
SRE: $ W_{1,t}
= A_{11,t}W_{1,t-1}+A_{12,t}W_{2,t-1}+B_{1,t}$ where  $W_{1,t}$ is influenced by $W_{2,t}$\footnote{In case 
$\alpha_1\neq\alpha_2$, possibility of
different tail indices was already suggested in Matsui and Mikosch
\cite{matsui:mikosch:2016}.}. The same dependency holds in any dimension
since $W_i$ depends on $W_{i+1},\ldots,W_d$ in a similar but more
complicated manner. 
In order to prove \eqref{assymp} we clarify this dependency with new
notions and apply an induction scheme effectively.
 
The structure of the paper is as follows. 
In Section \ref{sec:pre_statio}, we state preliminary
assumptions and we show existence of stationary solutions to
SRE \eqref{sre1}.
Section \ref{mainthproof} contains the main theorem together with its proof. 
The proof follows by induction and it
requires several preliminary results.
They are described
in Sections \ref{sec:case1} and \ref{sec:case2}.
 Applications to GARCH
models are suggested in Section
\ref{application}. Section \ref{conclusions} contains discussion about constants of tails and open problems related to the subject of the paper.

\subsection{Kesten's condition and previous results}
We state Kesten's result and briefly review
the literature. 
We prepare a function 
\[
h(s)=\inf_{n\in \N }(\E \| \bfPi_n \| ^s)^{1\slash n} \quad \text{with}\quad 
\bfPi_n = \bfA _1\cdots \bfA_n,
\]
where $\|\cdot\|$ is a matrix norm. 
The tail behavior of $\bfW$ is determined by $h$. We assume that there exists a unique $\alpha>0$ such that $h(\alpha)=1$.
The crucial assumption of Kesten is irreducibility. It can be described in several ways, neither of which seems simple and intuitive. We are going to state a weaker property, which is much easier to understand.
For any matrix $\bfY$, we write $\bfY>0$ if all entries are positive. The irreducibility assumption yields that for some $n\in\N$
\begin{align}\label{crutial}
\P(\bfPi_n >0)>0. 
\end{align}
Then the tail of $\bfW$ is essentially heavier than that of 
$\| \bfA \|$: it decays polynomially even if $\| \bfA \|$ is bounded. 
Indeed, there exists a function $e_{\alpha}$ on the unit sphere $\mathbb{S}^{d-1}$ such that 
\begin{equation} 
\label{rege}
\lim _{x\to \infty}x^{\alpha}\P (\bfy'\bfW>x)=e_{\alpha}(\bfy), \quad
\bfy \in \mathbb{S}^{d-1}
\end{equation}
and $e_{\alpha}(\bfy) > 0$ for $ y\in \mathbb{S}^{d-1}\cap
[0,\infty)^d$, where $\bfy'\bfW=\sum_{j=1}^d y_j W_j$ denotes the inner product.
If $\alpha \not\in \N$ then \eqref{rege} implies 
multivariate regular variation of $\bfW$, while if $\alpha\in \N$, the same holds under
some additional conditions (see
\cite[Appendix C]{buraczewski:damek:mikosch:2016}).
Here we say that a $d$-dimensional r.v. $\bfX$ is {\it multivariate regularly varying with index $\alpha$} if
 \begin{align}
 \label{eq:mrv1}
  \frac{\P(|\bfX|>ux,\,\bfX/|\bfX|\in \cdot)}{\P(|\bfX|>x)}
  \stackrel{v}{\to} u^{-\alpha}\P({\bf \Theta} \in \cdot),\quad u>0,
 \end{align}
 where $\stackrel{v}{\to}$ denotes vague convergence and $\bf\Theta$ is a
 random vector on the
 unit sphere $\mathbb{S}^{d-1}=\{\bfx\in \R^d:|\bfx| =1 \}$. This is a 
 common tool for describing multivariate power-tails, see \cite{bingham:goldie:teugels:1987,resnick:1987,resnick:2007} or 
\cite[p.279]{buraczewski:damek:mikosch:2016}. 

We proceed to the literature after Kesten. Later on an analogous result
was proved by Alsmeyer and Mentemeier \cite{alsmeyer:mentmeier2012} for invertible
matrices $\bfA$ with some irreducibility and density conditions (see also \cite{kluppelberg:pergamentchikov:2004}).
The density assumption was removed by Guivarc'h and Le Page
\cite{guivarch:lepage:2015} who developed the most general approach to
\eqref{sre1} with signed $\bfA$ having possibly a singular
law. Moreover, their conclusion was stronger, namely they obtained
existence of a measure $\mu $ on $\R ^d $ being the weak limit of 
\begin{equation}
\label{regular}
 x^{\alpha}\P (x^{-1}\bfW\in \cdot ) \quad \mbox{when}\ x\to \infty, 
\end{equation}
which is equivalent to \eqref{eq:mrv1}.
As a related work, in \cite{buraczewski:damek:guivarch2009} the existence of the limit \eqref{regular} was proved 
under the assumption that $\bfA$ is a similarity\footnote{$A$ is a similarity if for every $x\in \R ^d$, $|Ax|=\| A\|\
|x|$.}.
See \cite{buraczewski:damek:mikosch:2016} for the details of Kesten's
result and other related results.

For all the matrices $\bfA$ considered above, we have the same
tail behavior in all directions, one of the reasons being a certain
irreducibility or homogeneity of the operations generated by the support of the law of $\bfA$. This is not the case for triangular
matrices, since \eqref{crutial} does not hold then. Research in such directions would be a next natural
step\footnote{see also \cite{buraczewski:damek:2010}, 
\cite{buraczewski:damek:guivarch2009} and
\cite[Appendix D]{buraczewski:damek:mikosch:2016} for diagonal matrices}.
In this paper we work out the case when the indices
$\alpha_1,\ldots\alpha_d>0$ satisfying $\E A_{ii}^{\alpha_i}=1$ are all different. 
However, there remain not a few problems being settled, e.g. what happens when they are not
necessarily different. This is not clear even in the case of $2\times 2$
matrices. A particular case $A_{11}=A_{22}>0$ and $A_{12} \in \R$ was
studied in \cite{damek:zienkiewicz:2017} where the result is 
\[
\P (W_1>x)\sim
\begin{cases}
Cx^{-\alpha_1}(\log x)^{\alpha_1}\quad&{\rm if}\quad \E A_{12}A_{11}^{\alpha_1-1}\neq 0,\\
Cx^{-\alpha_1}(\log x)^{\alpha_1/2}\quad&{\rm if}\quad \E A_{12}A_{11}^{\alpha_1-1}=0,
\end{cases}
\]
where $C>0$ is a constant which may be different line by line.
Our conjecture in $d\, (>1)$ dimensional case is that 
\[
\P (W_i>x)\sim Cx^{-\wt\alpha_i}\ell_i(x),\quad i=1,\ldots,d 
\]
for some slowly varying functions $\ell_i$, and to get optimal $\ell_i$'s
 would be a real future challenge.

\section{Preliminaries and Stationarity}
\label{sec:pre_statio}
We consider $d\times d$ random matrices $\bfA=[A_{ij}]_{i,j=1}^d$ and
$d$-dimensional random vectors $\bfB$ that satisfy the set of assumptions:\\
\begin{itembox}[l]{{\bf Condition} (T)} 
(T-1)\ $\P(\bfA\ge 0)=\P(\bfB\ge 0)=1$, \\
(T-2)\ $\P(B_i=0)<1$ for $i=1,\ldots,d$, \\
(T-3)\ $\bfA$ is upper triangular, i.e. $\P(A_{ij}=0)=1$ whenever $i>j$,\\
(T-4)\ There exist $\alpha_1,\ldots,\alpha_d$ such that $\E A_{ii}^{\alpha_i}=1$ for $i=1,\ldots,d$ and $\alpha_i\not=\alpha_j$ if $i\not=j$,\\
(T-5)\ $\E A_{ij}^{\alpha_i}<\infty$ for any $i,j\le d$,\\
(T-6)\ $\E B_i^{\alpha_i}<\infty$ for $i=1,\ldots,d$,\\
(T-7)\ $\E [A_{ii}^{\alpha_i}\log^+ A_{ii}]<\infty$ for $i=1,\ldots,d$,\\
(T-8)\ The law of $\log A_{ii}$ conditioned on $\{A_{ii}>0\}$ is
 non-arithmetic for $i=1,\ldots,d$. 
\end{itembox}
\vspace{1mm}

\noindent
Note that (T-1) and (T-4) imply that 
\begin{equation}
\label{positive}
\P(A_{ii}>0)>0
\end{equation}
for $i=1,\ldots,d$. For further convenience any r.v. satisfying the inequality \eqref{positive} will be called {\it positive}. 
Most conditions are similar to those needed for applying Kesten-Goldie's
result (see \cite{buraczewski:damek:mikosch:2016}).

Let $(\bfA_t,\bfB_t)_{t\in\Z}$ be
an i.i.d. sequence with the generic element $(\bfA,\bfB)$. We define the products
\[
\bfPi_{t,s}=
\begin{cases}
\bfA_t\cdots\bfA_s,\quad &t\ge s,\\
{\bf I}_d,&t<s,
\end{cases}
\]
where ${\bf I}_d$ denotes the identity $d\times d$ matrix. In the case $t=0$ we are going to use a simplified notation
\[
\bfPi_n:=\bfPi_{0,-n+1},\quad n\in\N,
\]
so that $\bfPi_0={\bf I}_d$. For any $t\in\Z$ and $n\in\N$, the products $\bfPi_{t,t-n+1}$ and $\bfPi_n$ have the same distribution.

Let
$\Vert\cdot\Vert$ be the operator norm of the matrix: $\Vert{\bf
M}\Vert=\sup_{|{\bf x}|=1}|{\bf Mx}|$, where $|\cdot|$ is the Euclidean
norm of a vector. 
The top Lyapunov exponent associated with $\bfA$ is defined by 
\[
\gamma_{\bfA}=\inf_{n\ge 1}\frac{1}{n}\E[\log\Vert\bfPi_n\Vert].
\]
Notice that in the univariate case $\bfA=A\in\R$ and $\gamma _{\bfA}=\E [\log |A|]$.

If $\gamma _{\bfA}<0$  
then the equation
\begin{equation}
\label{affine}
\bfW\stackrel{d}=\bfA\bfW+\bfB
\end{equation}
has a unique solution $\bfW$ which is independent of $(\bfA,\bfB)$. Equivalently, the stochastic recurrence equation
\begin{equation}
\label{sre}
\bfW_t=\bfA_t\bfW_{t-1}+\bfB_t
\end{equation}
has a unique stationary solution and $\bfW_t\stackrel{d}=\bfW$. Then we can write the stationary solution as a series
\begin{equation}
\label{series}
\bfW_t=\sum_{n=0}^\infty \bfPi_{t,t+1-n}\bfB_{t-n}.
\end{equation}
Indeed, it is easily checked that the process $(\bfW_t)_{t=0}^\infty$ defined by \eqref{series} is stationary and solves \eqref{sre}, if the series on the right hand side of \eqref{series} is convergent for any $t\in\N$. The convergence is ensured if the top Lyapunov exponent is negative (for the proof see \cite{bougerol:picard:1992a}).
Negativity of $\gamma_\bfA$ follows from condition
(T). We provide a proof in Appendix \ref{sec:lyapunov}.

From now on, $\bfW=(W_1,\ldots,W_d)$ will always denote the solution to \eqref{affine} and by $\bfW_t=(W_{1,t},\ldots,W_{d,t})$ we mean the stationary solution to \eqref{sre}. Since they are equal in distribution, when we consider distributional properties, we sometimes switch between the two notations, but this causes no confusion.

\section{The main result and the proof of the main part}
\label{mainthproof}
\subsection{The main result}
We are going to determine the component-wise
tail indices of the solution $\bfW$
to stochastic equation \eqref{affine}.
Since $\bfA $ is upper triangular, the tail behavior of $W_{i}$ is affected only by $W_{j}$ for $j>i$, but not necessarily all of them: we allow some entries of $\bfA$ above the diagonal to vanish $a.s.$ In the extreme case of a diagonal matrix the tail of any coordinate may be determined independently of each other. On the other hand, if $\bfA$ has no zeros
above the diagonal, every coordinate is affected by all the subsequent ones. In order to describe this phenomenon precisely, we define a partial order relation on the set of coordinates. It clarifies
interactions between them.
\begin{definition}
\label{direct_dep}
For $i,j\in\{1,\ldots,d\}$ we say that $i$ directly depends on $j$ and
 write $i\preccurlyeq j$ if $A_{ij}$ is positive (in the sense of \eqref{positive}). We further write $i\prec j$ if $i\preccurlyeq j$ and $i\neq
j$.
\end{definition}
Observe that $i\preccurlyeq j$ implies $i\le j$ since $\bfA$ is upper
triangular, while $i\preccurlyeq i$ follows from the positivity of
diagonal entries. 
We extract each component of SRE \eqref{sre} and may write 
\begin{equation}
\label{coordinateSRE}
W_{i,t}=\sum_{j=1}^d A_{ij,t}W_{j,t-1}+B_{i,t}=\sum_{j:j\succcurlyeq
i}A_{ij,t}W_{j,t-1}+B_{i,t},
\end{equation}
where in the latter sum all coefficients $A_{ij,t}$ are positive. From this, we obtain the component-wise SRE 
in the form
\begin{equation}
\label{coordinate}
W_{i,t}=A_{ii,t}W_{i,t-1}+\sum_{j:j\succ i}A_{ij,t}W_{j,t-1}+B_{i,t}=A_{ii,t}W_{i,t-1}+D_{i,t},
\end{equation}
where
\begin{equation}
\label{D-terms}
D_{i,t}=\sum_{j:j\succ i}A_{ij,t}W_{j,t-1}+B_{i,t}. 
\end{equation}
Therefore, the tail of $W_{i,t}$
is determined by the comparison of the autoregressive behavior, characterized by the index $\alpha_i$, and the
tail behavior of $D_{i,t}$, which depends on the indices
$\alpha_{j},\,j\succ i$. To clarify this, we define recursively new exponents $\wt{\alpha}_i$, where $i$ decreases from $d$ to 1:
\begin{equation}
\label{newindices}
\wt\alpha_i=\alpha_i\land\min\{\wt\alpha_{j}: i\prec j\}.
\end{equation}
If there is no $j$ such that $i\prec j$, then we set $\wt\alpha_i=\alpha_i$. In particular, $\wt\alpha_d=\alpha_d$.
Notice that in general $\wt \alpha_i \neq \min\{\alpha_j: i\le j\}$. Depending on zeros of $\bfA$, two 
relations $\wt
\alpha_i = \min\{\alpha_j: i\le j\}$ and 
$\wt \alpha_i > \min\{\alpha_j: i\le j\}$
are possible for any $i<d$, see Example \ref{ex1}. 

For further convenience we introduce also a modified, transitive version of relation
$\preccurlyeq $. 
\begin{definition}
\label{indirect_dep}
We say that $i$ depends on $j$ and write $i\trianglelefteq j$ if there
 exists $m\in\N$ and a sequence $(i(k))_{0\le k\le m}$ such that $i=i(0)\le \cdots\le
 i(m)=j$ and $A_{i(k)i(k+1)}$ is positive for $k=0,\ldots,m-1$. We write $i\vartriangleleft j$ if $i\trianglelefteq j$ and $i\neq j$.
\end{definition}
Equivalently, the condition on the sequence $(i(k))_{0\le k\le m}$ can be presented in the form
$i\preccurlyeq i(1)\preccurlyeq\cdots\preccurlyeq i(m)=j$.
In particular, $i\preccurlyeq j$ implies $i\trianglelefteq j$. Now we can write
\begin{equation}
\label{triangle}
\wt\alpha_i=\min\{\alpha_j:i\trianglelefteq j \}
\end{equation}
which is equivalent to \eqref{newindices}, though has a more convenient form.

Although definitions \ref{direct_dep} and \ref{indirect_dep} are quite similar, there is a significant difference between them.
To illustrate the difference, we introduce the following notation for the entries of $\bfPi_{t,s}$:
\[
(\bfPi_{t,s})_{ij}=:\pi_{ij}(t,s).
\]
If $t=0$, we use the simplified form
\[
\pi_{ij}(n):=\pi_{ij}(0,-n+1),\quad n\in\N. 
\]
By $i\preccurlyeq j$ we mean that the entry $A_{ij}$ of the matrix $\bfA$ is positive, while
$i\trianglelefteq j$ means that for some $n\in\N$ the corresponding entry $\pi_{ij}(n)$ of the matrix $\bfPi_{n}$ is positive.
The former relation gives a stronger condition. On the other hand, the
latter is more convenient when products of matrices are considered,
especially when transitivity plays a role. Example \ref{ex1} gives a deeper insight into the difference between the two relations.
Throughout the paper both of them are exploited.

Now we are ready to formulate the main theorem.
\begin{theorem}
\label{main}
Suppose that condition {\rm(T)} is satisfied for a random matrix $\bfA$. Let $\bfW$ be the solution to \eqref{affine}. Then there exist strictly positive constants $C_1,\ldots,C_d$ such that
\begin{equation}
\label{mainthm}
\lim\limits_{x\to\infty}x^{\wt\alpha_i}\P(W_i>x)=C_i,\quad i=1,\ldots,d.
\end{equation}
\end{theorem}

The following example 
gives some intuition of what the statement of the theorem means in practice. 
\begin{example}\label{ex1}
Let
\[
\bfA=
\left(
\begin{array}{ccccc}
A_{11}&A_{12}&0&0&A_{15}\\
0&A_{22}&0&A_{24}&0\\
0&0&A_{33}&0&A_{35}\\
0&0&0&A_{44}&0\\
0&0&0&0&A_{55}
\end{array}
\right), 
\]
where $\P(A_{ij}>0)>0$ and other components are zero $a.s.$ Suppose that
 $\alpha_4<\alpha_3<\alpha_2<\alpha_5<\alpha_1$ and $\bfA$ satisfies the
 assumptions of Theorem \ref{main}. 
 Then $\wt \alpha_1=\alpha_4$, $\wt \alpha_2=\alpha_4$,
 $\wt\alpha_3=\alpha_3$, $\wt \alpha_4=\alpha_4$ and
 $\wt\alpha_5=\alpha_5$. 
\end{example}
We explain the example step by step. 
The last coordinate $W_5$ is the solution to 1-dimensional SRE, so its tail index is $\alpha_5$. 
Since there is no $j$ such that $4 \vartriangleleft j$, $\alpha_4$ is the tail index of $W_4$.
For the third coordinate the situation is different: we have
$3\vartriangleleft 5$, so the tail of $W_3$ depends on $W_5$. But
$\alpha_3<\alpha_5$, so the influence of $W_5$ is negligible and we
obtain the tail index $\wt\alpha_3=\alpha_3$.
Inversely, the relations $2\vartriangleleft 4 $ and $\alpha_2>\alpha_4$ imply $\wt\alpha_2=\alpha_4$.
The first coordinate clearly depends on the second and fifth, but recall
that the second one also depends on the fourth. Hence we have to compare
$\alpha_1$ with $\alpha_2,\alpha_4$ and $\alpha_5$. The smallest one is
$\alpha_4$, hence $\wt\alpha_1=\alpha_4$ is the tail index. Although the
dependence of $W_1$ on $W_4$ is indirect, we see it in the relation $1\vartriangleleft 4$.\\

\subsection{Proof of the main result}
\label{proof:main}
Since the proof includes several 
preliminary results which are long and technical, they are postponed to Sections \ref{sec:case1} and \ref{sec:case2}. To make the argument more readable
we provide an outline of the proof here, referring to 
those auxiliary results. In the proof of Theorem
\ref{main} we fix the coordinate number $k$ and
consider two cases: $\wt \alpha_k=\alpha_k$ and $\wt \alpha_k<\alpha_k$.

In the first case $\wt \alpha_k=\alpha_k$ the proof is based on
Goldie's result \cite[Theorem 2.3]{goldie:1991} and it is contained in
Lemmas \ref{small_moment} and \ref{minimal_alpha}.

If $\wt \alpha_k<\alpha_k$, the proof is more complicated and
requires several auxiliary results. We proceed by induction. Let
$j_0$ be the maximal coordinate among $j
\vartriangleright k$ such that the modified tail indices $\wt\alpha_j$
and $\wt\alpha_k$ are equal. 
Then clearly, $\wt \alpha_k=\alpha_{j_0}=\wt\alpha_\ell$
for any $\ell$ such that $k \trianglelefteq \ell \vartriangleleft j_0$. 
We start with the maximal among such coordinates (in the standard order on $\N$) and prove \eqref{mainthm} for it. 
Inductively we reduce 
$\ell$ using results for larger tail indices
and finally we reach $\ell=k$. 

We develop a component-wise series representation \eqref{eq:series} for $\bfW_0$. In Lemma \ref{series_rep} we prove that it indeed coincides a.s.\ with \eqref{series}.

The series \eqref{eq:series} is decomposed into parts step by step (Lemmas
\ref{q-decomposition}, \ref{one_many} and \ref{one_one}). Finally we obtain the following expression:
\begin{equation}
\label{main:expression}
W_{\ell,0} = \pi_{\ell j_0}(s)W_{j_0,-s}+R_{\ell j_0}(s).
\end{equation}
Our goal is to prove that as $s\to\infty$, the tail asymptotics of the first term $\pi_{\ell j_0}(s)W_{j_0,-s}$ approaches that of \eqref{mainthm}, while the second term $R_{\ell j_0}(s)$ becomes negligible.

The quantity $R_{\ell j_0}(s)$ is defined inductively by
\eqref{R} and \eqref{eq:remainer} as a finite collection of negligible
parts, each being estimated in a different way. In the process of
finding the main dominant term, we simultaneously
settle the upper bounds for the negligible parts. This is done through \eqref{qs2} in Lemma \ref{q-decomposition}, Lemma \ref{q^s_step} and 
\eqref{d2} under conditions of Lemma \ref{one_many} and Lemma
\ref{one_one}. Then in the end $R_{\ell j_0}(s)$ is proved to be negligible
as $s\to \infty$.

The final step is related to Breiman's lemma applied to $\pi_{\ell j_0}(s)W_{j_0,-s}$. 
By independence between $\pi_{\ell j_0}(s)$ and $W_{j_0,-s}$, 
we obtain the equality in asymptotics
\[
\P(\pi_{\ell j_0}(s)W_{j_0,-s} >x)\sim \E
[\pi_{\ell j_0}(s)^{\alpha_{j_0}}]\P(W_{j_0,-s}>x) 
\]
for fixed $s$. 
Recall the we need to let $s\to\infty$ for $R_{\ell j_0}(s)$ to be negligible. The existence of
$\lim_{s\to\infty} \E [\pi_{\ell j_0}(s)^{\alpha_{j_0}} ]$ is assured in
Lemma \ref{u:finite}. Then eventually we get $\P(W_{\ell,0}>x)\sim C\cdot
\P(W_{j_0,-s}>x)$ for a positive constant $C$, which can be explicitly computed. Now \eqref{mainthm} follows from stationarity and Lemma \ref{minimal_alpha}.

We move to the proof of Theorem \ref{main}.

\begin{proof}[Proof of Theorem \ref{main}]
Fix $k\le d$. If $\wt\alpha_k=\alpha_k$, then the statement of the
 theorem directly follows from Lemma \ref{minimal_alpha}. 
 If $\wt\alpha_k<\alpha_k$, then there is a unique $j_0\vartriangleright k$ such that $\wt\alpha_k=\wt\alpha_{j_0}=\alpha_{j_0}$. Another application of Lemma \ref{minimal_alpha} proves that
\begin{equation}
\label{j_0-tail}
\lim_{x\to\infty}x^{\alpha_{j_0}}\P(W_{j_0}>x)=C_{j_0}
\end{equation}
for some constant $C_{j_0}>0$. 
By stationarity we set $t=0$ without loss of generality.

The proof follows by induction with respect to number $j$ in backward direction, namely we start
 with $\ell=j_1:=\max\{j : k \trianglelefteq j \vartriangleleft j_0\}$ and reduce $\ell$
 until $\ell=k$. Notice that $j_1\prec j_0$, thus $\ell=j_1$ satisfies conditions of Lemma \ref{one_one}. Thus there exists $s_0$ such that for any $s>s_0$ we have 
\begin{align}
\label{eq:sre:mainpf}
W_{\ell,0} = \pi_{\ell j_0}(s)W_{j_0,-s}+R_{\ell j_0}(s),
\end{align}
where $\pi_{\ell j_0}(s)$ is independent of
 $W_{j_0,-s}$ such that $0<\E\pi_{\ell j_0}(s)^{\alpha_{j_0}}<\infty$, and $R_{\ell j_0}(s)$ satisfies \eqref{d2}. 
We are going to estimate $\P(W_{\ell,0}>x)$ from both below and above.
Since $R_{\ell j_0}(s)\ge 0\,a.s.$ 
\[
\P(W_{\ell,0}>x)\ge\P(\pi_{\ell j_0}(s)W_{j_0,-s}>x)
\]
holds, and by Breiman's lemma \cite[Lemma
 B.5.1]{buraczewski:damek:mikosch:2016} for fixed $s>s_0$
\[
\lim_{x\to\infty}\frac{\P(\pi_{\ell j_0}(s)W_{j_0,-s}>x)}{\P(W_{j_0,-s}>x)}
 =\E[\pi_{\ell j_0}(s)^{\alpha_{j_0}}] =:u_\ell(s).
\]
Combining these with \eqref{j_0-tail}, we obtain the lower estimate 
\begin{equation}
\label{liminf}
 \varliminf_{x\to\infty} x^{\alpha_{j_0}} \P(W_{\ell,0}>x)\ge u_\ell(s)\cdot
 C_{j_0}. 
\end{equation}
Now we pass to the upper estimate. 
Recall \eqref{d2} in Lemma \ref{one_one} which implies that 
for any $\delta\in(0,1)$ and $\varepsilon>0$ there exists $s_1\,(> s_0)$ such that for
 $s> s_1$ 
\[
 \varlimsup_{x\to\infty}x^{\alpha_{j_0}}\P\left(R_{\ell j_0}(s)>
 \delta x\right) = \delta^{-\alpha_{j_0}}\varlimsup_{x\to\infty}(\delta x)^{\alpha_{j_0}}\P\left(R_{\ell j_0}(s)>\delta x\right)<\varepsilon.
\]
Then for fixed $s\ge s_1$, we apply Lemma \ref{breiman_limsup}, which is a version of Breiman's lemma, to \eqref{eq:sre:mainpf} and obtain 
\begin{align}
\varlimsup_{x\to\infty}x^{\alpha_{j_0}}\P(W_{\ell,0}>x)&=\varlimsup_{x\to\infty}x^{\alpha_{j_0}}
\P\left(\pi_{\ell j_0}(s)W_{j_0,-s}+R_{\ell j_0}(s)>x\right) 
\nonumber \\
&\le\varlimsup_{x\to\infty}x^{\alpha_{j_0}}\P\left(\pi_{\ell j_0}(s)W_{j_0,-s}>x
 (1-\delta)\right)+\varlimsup_{x\to\infty}x^{\alpha_{j_0}}\P\left(R_{\ell
 j_0}(s)>\delta
 x\right) \nonumber \\
&<\varlimsup_{x\to\infty}x^{\alpha_{j_0}}\P\left(\frac{\pi_{\ell j_0}(s)}{1-\delta}W_{j_0,-s}>x\right)+\varepsilon \nonumber \\
&=\E[\pi_{\ell j_0}(s)^{\alpha_{j_0}}](1-\delta)^{-\alpha_{j_0}}\lim_{x\to\infty}x^{\alpha_{j_0}}\P(W_{j_0,-s}>x)+\varepsilon
 \nonumber \\
&=(1-\delta)^{-\alpha_{j_0}}u_\ell (s)\cdot C_{j_0}+\varepsilon, \label{limsup}
\end{align}
where we also use \eqref{j_0-tail}. We may let
 $\varepsilon\to 0$ and $\delta\to0$ together with $s\to\infty$ here and in \eqref{liminf}. The existence and positivity of 
 the limit 
$u_\ell =\lim_{s\to \infty} u_\ell (s)$ is assured by Lemma \ref{u:finite}. 
Thus from \eqref{liminf} and \eqref{limsup} we have  
\begin{align*}
 u_\ell \cdot C_{j_0} = \lim_{s\to \infty} u_\ell (s)\cdot C_{j_0} \le
 \varliminf_{x\to\infty}x^{\alpha_{j_0}}\P(W_{\ell,0}>x) \le
 \varlimsup_{x\to\infty}x^{\alpha_{j_0}}\P(W_{\ell,0}>x) \le \lim_{s\to \infty}
 u_\ell (s)\cdot C_{j_0} = u_\ell \cdot C_{j_0}. 
\end{align*}
This implies that 
\begin{equation}
\label{c_j}
\lim_{x\to\infty}x^{\wt\alpha_\ell}\P(W_{\ell ,0}>x)=u_\ell \cdot C_{j_0}=:C_\ell.
\end{equation}
Now we go back to the induction process.

If $j_1=k$, then the proof is over, and if $j_1\neq k$, we set
 $j_2=\max\{j:k\trianglelefteq j\vartriangleleft j_0, j\neq j_1\}$. Then
there are two possibilities, depending on whether $j_2\prec j_1$
or not. If $j_2\not\prec j_1$, then the assumptions of Lemma \ref{one_one} are satisfied with $\ell=j_2$ and we repeat the argument that we used for $j_1$. If $j_2\prec j_1$, the assumptions of Lemma \ref{one_many} are 
 fulfilled with $\ell=j_2$. Since the assertion of Lemma \ref{one_many}
 is the same as that of Lemma \ref{one_one}, we can again repeat the argument that we used for $j_1$. 

In general, we define 
$j_{m+1}=\max(\{j: k\trianglelefteq j \vartriangleleft
j_0\}\setminus\{j_1,\ldots,j_m\})$.
If $j_{m+1}\prec j$ for some $j\in\{j_1,\ldots,j_m\}$, then we use Lemma \ref{one_many} with $\ell=j_{m+1}$. Otherwise we use Lemma \ref{one_one} with the same $\ell$.
 Then by induction we prove \eqref{c_j} for every $\ell:k\trianglelefteq \ell\vartriangleleft j_0$, particularly in the end we also obtain
\[
C_k=\lim_{x\to\infty}x^{\alpha_{j_0}}\P(W_k>x).
\]
\end{proof}

Notice that we have two limit operations, with respect to $x$ and $s$,
and always the limit with respect to $x$ precedes. We cannot exchange the
limits, namely we have to let $s\to\infty$ with $s$ depending on $x$. 


\section{Case $\wt\alpha_k=\alpha_k$}\label{sec:case1}
The assumption $\wt\alpha_k=\alpha_k$ means that the tail behavior of
$W_k$ is determined by its auto-regressive property, namely the tail
index is the same as that of the solution $V$ of the
stochastic equation $V\stackrel{d}=A_{kk}V+B_k$. The tails of other coordinates on which $k$
depends are of smaller order, which is rigorously shown in Lemma \ref{small_moment}. In Lemma \ref{minimal_alpha} 
we obtain the tail index of $W_k$ by applying Goldie's result. 
By Lemma \ref{small_moment} we
observe that the perturbation induced by random elements other than those
of $k$th coordinate $(A_{kk},B_k)$ are negligible.

In what follows we work on a partial sum of the
stationary solution \eqref{series} component-wisely.  
To write the coordinates of the partial sum directly, the following definition is useful.   
\begin{definition}
\label{sequences}
For $m\in\N$ and $1\le i,j \le d$ let $H_m(i,j)$ be the set of all sequences of $m+1$ indices
 $(h(k))_{0\le k\le m}$ such that $h(0)=i, h(m)=j$ and $h(k)\preccurlyeq h(k+1)$ for $k=0,\ldots,m-1$.
For convenience, 
the elements of $H_m(i,j)$ will be denoted by $h$.
\end{definition}
Notice that each of such sequences is non-decreasing since $\bfA$ is upper triangular. Moreover, $H_m(i,j)$ is nonempty if and only if $i\trianglelefteq j$ and $m$ is large enough.
Now we can write
\begin{equation}
\label{pi}
 \pi_{ij}(s)=\sum_{h\in H_s(i,j)}\prod_{p=0}^{s-1}A_{h(p)h(p+1),-p}.
\end{equation}
Similar expression for $\pi_{ij}(t,s)$ can be obtained by shifting the time indices, 
\begin{equation*}
 \pi_{ij}(t,s) = \sum_{h \in H_{t-s+1}(i,j)} \prod_{p=0}^{t-s}
  A_{h(p)h(p+1),t-p},\quad t\ge s,
\end{equation*}
which will be used later.
Since the sum is finite, it follows from condition (T-5) that $\E
\pi_{ij}(s)^{\wt\alpha_i}<\infty$ for any $i,j$. Moreover if $j\trianglerighteq i$, then
$\P(\pi_{ij}(s)>0)>0$ for $s$
large enough (in particular $s\ge j-i$ is sufficient). 
By definition
$\pi_{ij}(s)$ is independent of $W_{j,-s}$. Notice that when $i=\ell$, $j=j_0$ and $\ell \trianglelefteq j_0$, it is the coefficient of our targeting representation 
\eqref{main:expression}. 
\begin{lemma}
\label{small_moment}
For any coordinate $i\in\{1,\ldots,d\}$, 
$\E W_i^\alpha<\infty$ if $
 0<\alpha<\wt\alpha_i$.
\end{lemma}
\begin{proof}
 For fixed $i$, let us approximate $\bfW_{0}$ by partial sums of the series \eqref{series}. We will denote $\bfS_n=\sum_{m=0}^{n}\bfPi_{m}\bfB_{-m}$ and $\bfS_n=(S_{1,n},\ldots,S_{d,n})$. We have then
 \begin{equation}
  \label{eq:defSin}
  S_{i,n}=\sum_{m=0}^n\sum_{j:j\trianglerighteq i}\pi_{ij}(m)B_{j,-m}
  =\sum_{m=1}^n\sum_{j:j\trianglerighteq i}\sum_{h\in H_m(i,j)}\Big(\prod_{p=0}^{m-1}A_{h(p)h(p+1),-p}\Big)B_{j,-m}+B_{i,0}.
 \end{equation}
 Suppose that $\alpha \le 1$. Then, by independence of
 $A_{h(p-1)h(p),-p},\,p=1,\ldots,m-1$ and $B_{j,-m}$, 
\[
 \E S_{i,n}^\alpha \le\sum_{m=1}^n\sum_{j:j \trianglerighteq i}\sum_{h\in H_m(i,j)}
 \Big(\prod_{p=0}^{m-1}\E A_{h(p)h(p+1),-p}^\alpha\Big) \E B_{j,-m}^\alpha +\E B_{i,0}^\alpha.
\]
To estimate the right-hand side, we will need to estimate the number of
 elements of $H_m(i,j)$.
 To see the convergence of the series
\eqref{eq:defSin} it suffices to consider $m>2d$. Recall that the
 sequences in $H_m(i,j)$ are non-decreasing, thus for a fixed $j$ there are at most $j-i$
 non-diagonal terms in each product on the right-hand side of \eqref{eq:defSin}. The non-diagonal terms in the product coincide with the time indices $t$, for which the values $h(t-1)$ and $h(t)$ are different.
If there are exactly $j-i$ such time indices, then
the values are uniquely determined. There are $\binom{m}{j-i}$
possibilities of placing the moments among $m$ terms of the sequence and $\binom{m}{j-i}<\binom{m}{d}$ since $m>2d$ and $j-i<d$.\par
 If we have $l<j-i$  non-diagonal elements in the product, then there are $\binom{m}{l}$ possibilities of placing them among other terms and there are less than $\binom{j-i}{l}$ possible sets of $l$ values. 
 Hence we have at  most $\binom{m}{l}\binom{j-i}{l}<\binom{m}{d}d!$ sequences for a fixed
 $l$. Moreover, there are $j-i< d$ possible values of $l$, and hence
 there are at most $d\cdot d!\cdot\binom{m}{d}$ sequences in
 $H_m (i,j)$. Since $\binom{m}{d}<\frac{m^d}{d!}$, the number of
 sequences in $H_m(i,j)$ is further bounded by $dm^d$. 

Now recall that there is $\rho<1$ such that $\E A_{jj}^\alpha\le\rho$
 for each $j\vartriangleright i$ and that there is a uniform bound $M$
 such that $\E A_{jl}^\alpha<M$ and $\E B_j^\alpha<M$ whenever $j
 \trianglerighteq i$, for any $l$. 
It follows that
\begin{align}
\label{ineq:Sin}
\E S_{i,n}^\alpha \le&\, C+ \sum_{m=2d}^n\sum_{j:j \trianglerighteq i}\sum_{h\in H_m(i,j)}M^{d+1}\rho^{m-d}\\
\le&\, C + \sum_{m=2d}^n\sum_{j:j\trianglerighteq i}dM^{d+1}\rho^{-d}\cdot m^d
 \rho^m \nonumber\\
\le&\, C+ \sum_{m=2d}^n d^2 M^{d+1}\rho^{-d}\cdot m^d \rho^m \nonumber\\
=&\, C+ d^2M^{d+1}\rho^{-d} \sum_{m=2d}^nm^d\rho^m<\infty \nonumber
\end{align}
uniformly in $n$, with $C>0$, which is bounded from above. Hence there
 exists the limit $S_i=\lim_{n\to\infty}S_{i,n}$ $a.s.$ and $\E
 S_i^\alpha<\infty$. By \eqref{series} we have 
 $\bfW_{0}=\lim_{n\to\infty}\bfS_n\ a.s.$ and we conclude
 that $\E W_{i,0}^\alpha <\infty$.\par
If $\alpha>1$, then by Minkowski's inequality we obtain
\[
(\E S_{i,n}^\alpha)^{1/\alpha}\le\, C'+ \sum_{m=2d}^n\sum_{j:j \trianglerighteq
 i}\sum_{h\in H_m(i,j)}\Big(\prod_{p=0}^{m-1}(\E A_{h(p)h(p+1),-p}^\alpha
)^{1/\alpha}\Big)(\E B_{j,-m}^\alpha)^{1/\alpha} 
\]
with $C'>0$. 
The same argument as above shows the uniform convergence. Thus the
 conclusion follows.
\end{proof}

Suppose that we have $\wt\alpha_k=\alpha_k$. This implies that
$\alpha_k<\wt\alpha_j$ for each $j\vartriangleright k$ 
and hence, by Lemma \ref{small_moment}, $\E W_j^{\alpha_k}<\infty$. 
The next lemma proves the assertion of the main theorem in case $\wt \alpha_k=\alpha_k$.
\begin{lemma}
\label{minimal_alpha}
Suppose assumptions of Theorem \ref{main} are satisfied and let $k\le d$. Provided that
$\wt\alpha_k=\alpha_k$,
there exists a positive constant $C_k$ such that
\begin{equation}
\lim\limits_{x\to\infty}x^{\wt\alpha_k}\P(W_k>x)=C_k.
\end{equation}
\end{lemma}
\begin{proof}
We are going to use Theorem 2.3 of Goldie \cite{goldie:1991} which asserts that if
\begin{equation}
\label{goldie:integral}
\int_0^\infty|\P(W_{k}>x)-\P(A_{kk}W_{k}>x)|x^{\alpha_k-1}dx<\infty, 
\end{equation}
then
\begin{equation}
 \label{k-tail}
\lim_{x\to\infty}x^{\alpha_k}\P(W_k>x)=C_k,
\end{equation}
where
\begin{equation}
\label{c_integral}
C_k=\frac{1}{\E\big[A_{kk}^{\alpha_k}\log A_{kk}\big]}\int_0^\infty\left(\P(W_{k}>x)-\P(A_{kk}W_{k}>x)\right)x^{\alpha_k-1}dx.
\end{equation}
To prove \eqref{goldie:integral}, we are going to use Lemma 9.4 from
 \cite{goldie:1991}, which derives the equality 
\begin{equation}
\label{goldie:lemma}
\int_0^\infty |\P(W_{k,0}>x)-\P(A_{kk,0}W_{k,-1}>x)|x^{\alpha_k-1}dx=\frac{1}{\alpha_k}\E
\big|W_{k,0}^{\alpha_k}-(A_{kk,0}W_{k,-1})^{\alpha_k}\big|=:I.
\end{equation}
From \eqref{coordinateSRE} we deduce that $W_{k,0}\ge A_{kk,0}W_{k,-1}$ a.s. Hence the absolute value may be omitted on both sides of \eqref{goldie:lemma}.
We consider two cases depending on the value of $\alpha_k$. \\
{\bf Case 1. $\alpha_k<1$.}\par
Due to \eqref{coordinate} and Lemma \ref{small_moment} we obtain 
\[
\alpha_k I\le\E\left[(W_{k,0}-A_{kk,0}W_{k,-1})^{\alpha_k}\right]=\E
 D_{k,0}^{\alpha_k}<\infty. 
\]
{\bf Case 2. $\alpha_k\ge 1$.}\par
For any $x\ge y\ge 0$ and $\alpha\ge 1$ the following inequality holds:
\[
x^\alpha-y^\alpha=\alpha\int_y^x t^{\alpha-1}dt\le\alpha x^{\alpha-1}(x-y).
\]
Since $W_{k,0}\ge A_{kk,0}W_{k,-1}\ge 0$ a.s. we can estimate
\begin{align*}
 W_{k,0}^{\alpha_k}-(A_{kk,0}W_{k,-1})^{\alpha_k}&
\le \alpha_k W_{k,0}^{\alpha_k-1}(W_{k,0}-A_{kk,0}W_{k,-1})\\
&=\alpha_k D_{k,0}W_{k,0}^{\alpha_k-1}\\
&=\alpha_k D_{k,0}(A_{kk,0}W_{k,-1}+D_{k,0})^{\alpha_k-1}.
\end{align*}
Since the formula
\[
(x+y)^\alpha\le\max\{1,2^{\alpha-1}\}(x^\alpha+y^\alpha)
\]
holds for any $x,y>0$ and each $\alpha$, by putting $\alpha=\alpha_k-1$ we obtain
\[
 I\le\E\left[D_{k,0}(A_{kk,0}W_{k,-1}+D_{k,0})^{\alpha_k-1}\right]\le\max\{1,2^{\alpha_k-2}\}
\big(\E
 D_{k,0}^{\alpha_k}+\E [ D_{k,0}(A_{kk,0}W_{k,-1})^{\alpha_k-1} ]\big).
\]
Since $\E D_{k,0}^{\alpha_k}<\infty$ by Lemma \ref{small_moment}, it
 remains to prove the finiteness of the second expectation.
In view of \eqref{coordinate}, 
\begin{align*}
\E [ D_{k,0}(A_{kk,0}W_{k,-1})^{\alpha_k-1} ] =\E\left[A_{kk,0}^{\alpha
 _k-1}B_{k,0}\right]\E W_{k,-1}^{\alpha_k-1}+\sum_{j\succ
 k}\E\left[A_{kk,0}^{\alpha_k-1}A_{kj,0}\right]
\E\left[W_{k,-1}^{\alpha_k-1}W_{j,-1}\right],
\end{align*}
where we use independence of $(A_{\cdot\cdot,0},B_{\cdot,0})$ and
 $W_{\cdot,-1}$.
Since $\E W_{k,-1}^{\alpha_k-1}<\infty$ by Lemma \ref{small_moment}, we focus on the remaining terms. 
Take $p=\frac{\alpha_k+\varepsilon}{\alpha_k+\varepsilon-1}$ and
 $q=\alpha_k+\varepsilon$ with
 $0<\varepsilon<\min\{\wt\alpha_j:j\vartriangleright k\}-\alpha_k$. Then since
\[
 p(\alpha_k-1)=\frac{\alpha_k+\varepsilon}{\alpha_k+\varepsilon-1}(\alpha_k-1)
 =\left(1+\frac{1}{\alpha_k+\varepsilon-1}\right)(\alpha_k-1)<\left(1+\frac{1}{\alpha_k-1}\right)(\alpha_k-1)=\alpha_k, 
\]
 the H\"{o}lder's inequality together with Lemma \ref{small_moment} yields
\[
 \E\left[W_{k,-1}^{\alpha_k-1}W_{j,-1}\right]\le\left(\E
 W_{k,-1}^{p(\alpha_k-1)}\right)^{1/p}\cdot\left(\E
 W_{j,-1}^{\alpha_k+\varepsilon}\right)^{1/q}<\infty.
\]
Similarly $\E [A_{kk,0}^{\alpha_k-1} B_{k,0}]<\infty$ and
 $\E[A_{kk,0}^{\alpha_k-1}A_{kj,0}]<\infty$
 hold and hence $I<\infty$. 

 By \cite[Theorem 2.3]{goldie:1991}, \eqref{k-tail}
 holds and it remains to prove that $C_k>0$.
 Since $W_{k,0}\ge A_{kk,0}W_{k,-1}+B_{k,0}$ holds from \eqref{limsup},
 $W_{k,0}^{\alpha_k}-(A_{kk,0}W_{k,-1})^{\alpha_k}$ is strictly positive
 on the set $\{B_{k,0}>0\}$ which has positive probability in view of
 condition (T-2). Therefore in \eqref{goldie:lemma} we obtain
 $I>0$. Condition (T-7) implies that $0<\E[A_{kk}^{\alpha_k}\log
 A_{kk}]<\infty$. Hence we see from \eqref{c_integral} that $C_k>0$. 
\end{proof}


\section{Case $\wt\alpha_k<\alpha_k$}
\label{sec:case2}
The situation is now quite the opposite. The auto-regressive behavior of
$W_k$ does not play any role since $k$ depends (in terms of Definition \ref{indirect_dep}) on coordinates which admit
dominant tails. More precisely, we prove that $W_k$ has a regularly
varying tail and its tail index is smaller than $\alpha_k$. It is equal
to the tail index $\alpha_{j_0}$ of the unique component $W_{j_0}$ such
that $\wt\alpha_k= \alpha_{j_0}(=\wt \alpha_{j_0})$. 
The latter is due to the formula 
\begin{equation}\label{eq:decomp}
 W_{\ell,0} = \pi_{\ell j_0}(s)W_{j_0,-s}+R_{\ell j_0}(s),
\end{equation}
and it is proved inductively for all $\ell$ such that $k \trianglelefteq \ell \vartriangleleft j_0$. The decomposition \eqref{eq:decomp} is the main goal in this section. We show that 
$R_{\ell j_0}(s)$ is negligible as $s\to\infty$, so that the
tail of $W_{\ell,0}$ comes from $\pi_{\ell j_0}(s)W_{j_0,-s}$. Moreover, we apply Breiman's lemma to $\pi_{\ell j_0}(s)W_{j_0,-s}$.
Then we also need to describe the limit behavior of $\E \pi_{\ell
j_0}(s)^{\alpha_{j_0}}$ as $s\to\infty$. 

To reach \eqref{eq:decomp}, we first prove that $W_{\ell,0}$ may be represented as in \eqref{eq:series} below. Then we divide the series
\eqref{eq:series} into two parts: the finite sum $Q_F(s)$ and the tail
$Q_T(s)$ of \eqref{series_division}. Next, we decompose $Q_F(s)$ into two
parts: $Q_W(s)$ containing $W$-terms and $Q_B(s)$ containing $B$-terms,
see \eqref{qs1} and \eqref{q_s2}. Finally, $Q_W(s)$ is split into
$Q_W'(s)$ and $Q_W''(s)$, where the former contains the components with
dominating tails, while the latter gathers those with lower-order tails.
This decomposition suffices to settle the induction basis in Lemma \ref{one_one}, since \eqref{d1} is satisfied if we set $R_{\ell j_0}(s)=Q''_W(s)+Q_B(s)+Q_T(s)$.
The three parts $Q_T(s),Q_B(s)$ and $Q_W''(s)$
are estimated in Lemmas \ref{q^s_step} and \ref{q-decomposition}.

For the induction step in Lemma \ref{one_many} we find $s_1,s_2\in\N$ such that $s=s_1+s_2$ and extract another term $Q_W^*(s_1,s_2)$ from $Q'_W(s_1)$, so that $Q'_W(s_1)=\pi_{\ell j_0}(s)W_{j_0,-s}+Q_W^*(s_1,s_2)$. Then we set
\[
R_{\ell j_0}(s)=Q''_W(s_1)+Q_B(s_1)+Q_T(s_1)+Q_W^*(s_1,s_2).
\]
Notice that the two definitions of $R_{\ell j_0}(s)$ above coincide, since $Q_W^*(s,0)=0$ in the framework of Lemma \ref{one_one}.
The term $Q_W^*(s_1,s_2)$ is estimated in the induction step in Lemma \ref{one_many}, where both $s_1$ and $s_2$ are required to be sufficiently large.
The limit behavior of $\E\pi_{\ell j_0}(s)^{\alpha_{j_0}}$ as $s\to\infty$ is given in Lemma
\ref{u:finite}.

\subsection{Decomposition of the stationary solution}\label{decomp}

We need to introduce some notation.
For the products of the diagonal entries of the matrix $\bfA$ we write 
\[
\Pi_{t,s}^{(i)}=
\begin{cases}
A_{ii,t}\cdots A_{ii,s},\qquad&t\ge s,\\
1,\qquad&t<s,
\end{cases}\quad i=1,\ldots,d; t,s\in\Z
\]
with a simplified form for $t=0$
\[
\Pi_n^{(i)}:=\Pi_{0,-n+1}^{(i)},\quad n\in\N.
\]
We also define the subset $H'_m(i,j)\subset H_m(i,j)$ as 
\begin{equation}
\label{hprim}
H'_m(i,j)=\{h\in H_m(i,j):h(1)\neq i\}, 
\end{equation}
namely $H_m'(i,j)$ is the subset of $H_m(i,j)$ such that the first two terms of its elements are not equal: $i=h(0)\neq h(1)$. The latter naturally yields another definition
\begin{equation}
\label{piprim}
 \pi'_{ij}(t,t-s):=\sum_{h\in H'_{s+1}(i,j)}\prod_{p=0}^s A_{h(p)h(p+1),t-p};\qquad \pi_{ij}'(s):=\pi'_{ij}(0,-s+1),
\end{equation}
which looks similar to $\pi_{ij}(t,t-s)$ and can be interpreted as an
entry of the product $\bfA^0_t\bfPi_{t-1,t-s}$, where $\bfA^0$ stands
for a matrix that has the same entries as $\bfA$ outside the diagonal and zeros on the diagonal.

Now we are going to formulate and prove the auxiliary results of this section. We start with representing $W_{\ell,t}$ as a series.
\begin{lemma}
 \label{series_rep}
Let $\bfW_t=(W_{1,t},...,W_{d,t})$ be the stationary solution of \eqref{sre}. Then
\begin{equation}
 \label{eq:series}
 W_{i,t}=\sum_{n=0}^\infty \Pi^{(i)}_{t,t-n+1}D_{i,t-n}\quad a.s.,\quad i=1,\ldots,d,
\end{equation}
where $D_{i,t-n}$ is as in \eqref{D-terms}.
\end{lemma}
\begin{remark}
Representation \eqref{eq:series} allows us to
 exploit the fast decay of $\Pi^{(i)}_{t,t-n+1}$ as $n\to \infty$.
\end{remark}
\begin{proof}
First we show that for a fixed $\ell$ the series on the right hand side of \eqref{eq:series} converges.
For this we evaluate the moment of some order $\alpha$ from the range $0<\alpha<\wt \alpha_i$. Without loss of generality we let $\alpha<1$ and $t=0$. By Fubini's theorem 
\begin{align*}
 \E\left(\sum_{n=0}^\infty \Pi^{(i)}_{n} D_{i,-n}
\right)^\alpha
\le \sum_{n=0}^\infty \rho^n (dML+M)<\infty,
\end{align*}
where $L=\max\{\E W_j^\alpha:j\succ i\}$ and $M=\E B_i^\alpha\vee\max\{\E A_{ij}^\alpha:j\succ i\}$ are finite constants, and $\rho=\E A_{ii}^\alpha < 1$. Hence the series defined in \eqref{eq:series} converges.

Denote $Y_{i,t}:=\sum_{n=0}^\infty \Pi_{t,t-n+1}^{(i)}D_{i,t-n}$. We are going to show that the vector $\bfY_t=(Y_{1,t},\ldots,Y_{d,t})$ is a stationary solution to \eqref{sre}. Since stationarity is obvious, without loss of generality set $t=0$. For $i=1,\ldots,d$ we have
\begin{align}
\nonumber
 Y_{i,0}&=D_{i,0}+ \sum_{n=1}^\infty
 \Pi_{n}^{(i)} D_{i,-n} \\
\nonumber
 &= D_{i,0}+ A_{ii,0}\sum_{n=0}^\infty \Pi_{-1,-n}^{(i)} D_{i,-n-1}\\
\label{Y}
&=A_{ii,0}{Y_{i,-1}}+D_{i,0}.
\end{align}
The $d$ equations of the form \eqref{Y} (one for each $i$) can be written together as a matrix equation
\[
\bfY_0=\bfA_0\bfY_{-1}+\bfB_0,
\]
which is the special case of \eqref{sre}. It remains to prove that this implies that $Y_{i,0}=W_{i,0}$ a.s.

The series representation \eqref{series} allows to
 write $D_{i,t}$ as a measurable function of
 $\{(\bfA_s,\bfB_s):s<t\}$. Since the sequence
 $(\bfA_t,\bfB_t)_{t\in\Z}$ is i.i.d. and hence ergodic, it follows from
 Proposition 4.3 of \cite{krengel} that $(D_{i,t})_{t\in\Z}$ is also
 ergodic. Brandt \cite{brandt:1986} proved that then $Y_{i,t}$ defined in the paragraph above \eqref{Y} is the only proper solution to \eqref{eq:series} \cite[Theorem 1]{brandt:1986}. Therefore $Y_{i,t}=W_{i,t}$ a.s.
\end{proof}

Now we divide the series \eqref{eq:series} into the two parts: the finite sum $Q_F(s)$ of the $s$ first elements and the corresponding tail $Q_T(s)$:
\begin{equation}
\label{series_division}
 W_{\ell,0} =: \underbrace{\sum_{n=0}^{s-1}\Pi^{(\ell)}_{n}D_{\ell,-n}}_{Q_F(s)} + \underbrace{\sum_{n=s}^\infty
 \Pi^{(\ell)}_{n}D_{\ell,-n}}_{Q_T(s)}.
\end{equation}
For notational simplicity, we abbreviate coordinate number $\ell$ in $Q_F(s)$ and $Q_T(s)$. In the paper the appropriate coordinate is always denoted by $\ell$. The same applies for other parts $Q_W(s)$, $Q_B(s)$, $Q'_W(s)$, $Q''_W(s)$ and $Q_W^*(s_1,s_2)$ of $W_{\ell,0}$ which will be defined in the sequel.

 Both $Q_F(s)$ and $Q_T(s)$ have
 the same tail order for any $s$. However, if $s$ is large enough, one can prove that $Q_T(s)$ becomes negligible in the sense that the ratio of $Q_T(s)$ and $Q_F(s)$ tends to zero. The following lemma describes precisely that phenomenon. Recall
that $j_0$ is the unique index with the property 
$\wt\alpha_\ell=\alpha_{j_0}$ and so $\wt \alpha _j=\alpha _{j_0}$ for
all $j$ such that $\ell  \vartriangleleft j \trianglelefteq j_0 .$ 

\begin{lemma}
\label{q^s_step}
Assume that for any $j$ satisfying $j_0\trianglerighteq j\vartriangleright \ell$, 
\begin{equation}
\label{q^s_assumption}
\lim_{x\to\infty}x^{\alpha_{j_0}}\P(W_j>x)=C_j>0.
\end{equation}
Then for every $\varepsilon>0$ there exists $s_0=s_0(\ell,\varepsilon)$ such that 
\begin{equation}
\label{q^s_estimate}
\varlimsup_{x\to\infty}x^{\alpha_{j_0}}\P(Q_T(s)>x)<\varepsilon
\end{equation}
for any $s\ge s_0$.
\end{lemma}
\begin{proof}
For $\rho:=\E A_{\ell\ell}^{\alpha_{j_0}}<1$, we choose a constant
 $\gamma\in(0,1)$ such that
 $\rho\gamma^{-\alpha_{j_0}}=:\delta<1$. Notice that 
$ \E \big (\Pi^{(\ell)}_{n})^{\alpha_{j_0}}$ decays to $0$ exponentially fast, which is crucial in the following argument. We have
\begin{align*}
\P(Q_T(s)>x)=&\P\,\Big(\sum_{n=s}^\infty \Pi^{(\ell)}_{n} D_{\ell,-n}>x\Big)\\
\le&\sum_{n=s}^\infty\P\left(\Pi^{(\ell)}_{n} D_{\ell,-n}>x\cdot(1-\gamma)\gamma^{n-s}\right)\\
=&\sum_{n=0}^\infty\P\left(\Pi^{(\ell)}_{n+s}D_{\ell,-n-s}>x\cdot(1-\gamma)\gamma^n\right) \\
\le & \sum_{n=0}^\infty \P \left(\Pi^{(\ell)}_{n+s}
 D'_{\ell,-n-s}>x\cdot \frac{(1-\gamma)\gamma^{n}}{2}\right) (:=\mathrm{I})\\
& + \sum_{n=0}^\infty \P \left(\Pi^{(\ell)}_{n+s}
 D''_{\ell,-n-s}>x\cdot
 \frac{(1-\gamma)\gamma^{n}}{2}\right)(:=\mathrm{II}),  
\end{align*}
where we divided $D_{\ell,-n-s}=B_{\ell,-n-s}+\sum_{j\succ \ell}
 A_{\ell j,-n-s}W_{j,-n-s-1}$ into two parts,
\begin{equation}
\label{D-decomposition}
 D_{\ell,-n-s} = \underbrace{\sum_{j:\ell \prec j \trianglelefteq j_0}
 A_{\ell j,-n-s}W_{j,-n-s-1}}_{D'_{\ell,-n-s}} +
 \underbrace{\sum_{j:\ell \prec j \not\trianglelefteq j_0}
 A_{\ell j,-n-s}W_{j,-n-s-1}+B_{\ell,-n-s}}_{D''_{\ell,-n-s}}.
\end{equation}
The first part $D'_{\ell,-n-s}$ contains those components of $\bfW_{-n-s-1}$ which are dominant, while $D''_{\ell,-n-s}$ gathers the negligible parts: the components with lower-order tails and the $B$-term.
For the second sum $\mathrm{II}$, we have  $\E
 (D''_{\ell,-n-s})^{\alpha_{j_0}}<\infty$ by Lemma
 \ref{small_moment} and condition (T), so that Markov's
 inequality yields
\begin{align}
 x^{\alpha_{j_0}}\mathrm{II} \le \sum_{n=0}^\infty
 2^{\alpha_{j_0}}(1-\gamma)^{-\alpha_{j_0}} \gamma^{-n\alpha_{j_0}} \rho^{n+s} \E
 (D''_{\ell,-n-s})^{\alpha_{j_0}}
\le c\, \rho^{s} \sum_{n=0}^\infty \delta^n < \infty. \label{ineq:sumII}
\end{align}
For the first sum $\mathrm{I}$, we use conditioning in the following
 way.
\begin{align*}
 x^{\alpha_{j_0}} \mathrm{I} &\le \sum_{n=0}^\infty \sum_{j:\ell \prec j \trianglelefteq j_0} x^{\alpha_{j_0}} \P 
\left(
\Pi^{(\ell)}_{n+s}A_{\ell j,-n-s}W_{j,-n-s-1}>\frac{x\cdot(1-\gamma)\gamma^n}{2d}
\right) \\
& = \sum_{n=0}^\infty \sum_{j:\ell \prec j \trianglelefteq j_0} \E \big[
 x^{\alpha_{j_0}} \P\,\big(G_n W_{j,-n-s-1}>x \mid G_n \big)
\big],
\end{align*}
where
 $G_n=\Pi^{(\ell)}_{n+s}A_{\ell j,-n-s}(1-\gamma)^{-1}\gamma^{-n}\cdot
 2d$. Notice that $G_n$ and $W_{j,-n-s-1}$ are independent and
\[
 \E G_n^{\alpha_{j_0}}\le c_\ell\cdot \rho^{n+s}((1-\gamma)\gamma^n)^{-\alpha_{j_0}}(2d)^{\alpha_{j_0}},
\]
where $c_\ell = \max_{j}\{\E A_{\ell j}^{\alpha_{j_0}}: \ell \prec j \trianglelefteq j_0 \}$. Recall that by assumption \eqref{q^s_assumption}, there is a constant $c_j$ such that for every $x>0$ 
\[
 \P(W_j>x) \le c_j x^{-\alpha_{j_0}}.
\]
Therefore, recalling $\rho\gamma^{-\alpha_{j_0}}=:\delta$, we further obtain
\begin{align}
 x^{\alpha_{j_0}}\mathrm{I} \le \sum_{n=0}^\infty \sum_{j:\ell \prec j \trianglelefteq j_0 } c_j\cdot c_\ell\cdot
 \rho^{n+s}((1-\gamma)\gamma^n )^{-\alpha_{j_0}} (2d)^{\alpha_{j_0}}
 \le c' \cdot \rho^{s} \sum_{n=0}^\infty \delta^n \label{ineq:sumI}
\end{align}
with $c'=d\cdot c_\ell \cdot \max\{c_j:\ell \prec j
 \trianglelefteq j_0 \}\cdot(1-\gamma)^{-\alpha_{j_0}}(2d)^{\alpha_{j_0}}$. Now in view of \eqref{ineq:sumII}
 and \eqref{ineq:sumI}, since $\rho<1$, there is $s_0$ such that 
\[
\varlimsup_{x\to\infty}x^{\alpha_{j_0}}\P(Q_T(s)>x)<\varepsilon
\]
for $s>s_0$. 
\end{proof}

The dominating term $Q_F(s)$ of \eqref{series_division} can be further decomposed. 
\begin{lemma}
\label{q-decomposition}
For any $\ell \le d$ and $s\in\N$, the sum $Q_F(s)$ admits the decomposition
\begin{equation}
\label{decomposition}
Q_F(s)=Q_W(s)+Q_B(s)\quad a.s.
\end{equation}
where
\begin{align}
\label{qs1}
Q_W(s) &=\sum\limits_{j:j\vartriangleright \ell}\pi_{\ell j}(s)W_{j,-s}, \\
\label{q_s2}
Q_B(s)&=
\sum_{n=0}^{s-1}\Pi_{n}^{(\ell)}\Bigg(B_{\ell,-n}+\sum_{j:j\vartriangleright
 \ell}
\sum_{m=1}^{s-n-1}\pi_{\ell j}(-n,-n-m+1)B_{j,-n-m}\Bigg).
\end{align}
Moreover 
\begin{equation}
\label{qs2}
\E Q_B(s)^{\wt\alpha_\ell}<\infty.
\end{equation}
\end{lemma}
\begin{remark}
Each ingredient $D_{\ell , -n}=\sum_{j:j\succ \ell}A_{\ell
       j,-n}W_{j,-n-1}+B_{\ell,-n}$ of $Q_F(s)$ in \eqref{series_division} contains several
       $W$-terms with index $-n-1$ and a $B$-term with time index $-n$ $($see \eqref{D-terms}$)$. 
The idea is to change $W_{\cdot,-n-1}$ terms into $W_{\cdot,-s}$ by the iterative use of the recursion \eqref{sre} componentwise. Meanwhile, some additional $B$-terms, with time indices between $-n$ and
       $-s+1$, are produced. When all $W$-terms have the same time index
       $-s$, we gather all ingredients containing a $W$-term in
       $Q_W(s)$, while the remaining part $Q_B(s)$ consists of all
       ingredients containing a $B$-term. The quantity $Q_W(s)$ is easily treatable
       and $Q_B(s)$ is proved to be negligible in the tail (comparing to $Q_W(s)$, for which the lower bound is shown in the proof of Theorem \ref{main}).
\end{remark}
\begin{proof}
{\bf Step 1. Decomposition of $Q_F(s)$.}\par
In view of \eqref{series_division} and \eqref{coordinate} we have
\begin{align}
\label{qs:series}
Q_F(s)=\sum_{n=0}^{s-1}\Pi_{n}^{(\ell)}D_{\ell,-n}
=\sum_{n=0}^{s-1} \underbrace{\Pi_{n}^{(\ell)}\Big(\sum_{i:i\succ
 \ell}A_{\ell
 i,-n}W_{i,-n-1}+B_{\ell,-n}\Big)}_{\xi_n}:=\sum_{n=0}^{s-1}
 \xi_n. 
\end{align}
 We will analyze each ingredient $\xi_n$ of the last sum. We divide it into two parts, $\eta_n$ and $\zeta_n$, starting from $n=s-1$.
\begin{align*}
\xi_{s-1} &= \underbrace{\Pi_{s-1}^{(\ell)}\sum_{j:j\succ
 \ell}A_{\ell j,-s+1}W_{j,-s}}_{\eta_{s-1}}+\underbrace{\Pi^{(\ell)}_{s-1}B_{\ell,-s+1}\vphantom{\sum_j}}_{\zeta_{s-1}}
=\eta_{s-1}+\zeta_{s-1}.
\end{align*}
Next, for $n=s-2$ applying component-wise SRE \eqref{coordinateSRE}, we obtain
\begin{align*}
\xi_{s-2}&=\Pi_{s-2}^{(\ell)}\sum_{j:j\succ \ell}A_{\ell
 j,-s+2}\Big(\sum_{i:i\succcurlyeq j}A_{ji,-s+1}W_{i,-s}
+B_{j,-s+1}\Big)+\Pi^{(\ell)}_{s-2}B_{\ell,-s+2}\\
&=\underbrace{\Pi_{s-2}^{(\ell)}\sum_{j:j\succ
 \ell}A_{\ell j,-s+2}\sum_{i:i\succcurlyeq
 j}A_{ji,-s+1}W_{i,-s}}_{\eta_{s-2}}+
\underbrace{\Pi_{s-2}^{(\ell)}\Big(B_{\ell,-s+2}+\sum_{j:j\succ
 \ell}A_{\ell j,-s+2}B_{j,-s+1}\Big)}_{\zeta_{s-2}}=\eta_{s-2}+\zeta_{s-2}.
\end{align*}
In this way, for $n<s$ we define $\eta_n$ which consists of terms including
 $(W_{j,-s})$ and $\zeta_n$ which contains terms 
 $(B_{j,-i})$. Observe that in most $\eta$'s and
 $\zeta$'s, a multiple summation appears,
 which is not convenient. To write them in simpler forms, we are going
 to use the notation \eqref{hprim}.
This yields 
\begin{align}
\nonumber
\eta_{s-2}&=\Pi_{s-2}^{(\ell)}\sum_{i:i\vartriangleright
 \ell}\sum_{j:i\succcurlyeq j\succ \ell}A_{\ell j,-s+2}A_{ji,-s+1}W_{i,-s} 
\nonumber \\
&=\Pi_{s-2}^{(\ell)}\sum_{i:i\vartriangleright \ell}\sum_{h\in
 H'_2(\ell,i)}
\Big(\prod_{p=0}^1 A_{h(p)h(p+1),-s+2-p}\Big)W_{i,-s}\nonumber\\
&=\Pi_{s-2}^{(\ell)}\sum_{i:i\vartriangleright
 \ell}\pi'_{\ell
 i}(-s+2,-s+1)W_{i,-s}. 
\label{eta1}
\end{align}
For each $\eta_n$ we obtain an expression of a simple form similar to \eqref{eta1}.
To confirm this, let us return to the decomposition of $\xi$ and see one more step of the iteration for $\eta$. For $n=s-2$ we have 
\begin{align}
\xi_{s-2}
&=\Pi_{s-3}^{(\ell)}\Bigg[\sum_{j:j\succ \ell}A_{\ell
 j,-s+3}\Big\{\sum_{i:i\succcurlyeq j}A_{ji,-s+2}
\Big(\sum_{u:u\succcurlyeq i}A_{iu,-s+1}W_{u,-s}+B_{i,-s+1}\Big)+B_{j,-s+2}\Big\}+B_{\ell,-s+3}\Bigg],\label{xi2}
\end{align}
so that similarly to the case $n=s-2$, we obtain the expression 
\begin{align}
\nonumber
\eta_{s-3}&=\Pi_{s-3}^{(\ell)}\sum_{j:j\succ \ell}A_{\ell j,-s+3}
\sum_{i:i\succcurlyeq j}A_{ji,-s+2}\sum_{u:u\succcurlyeq i}A_{iu,-s+1}W_{u,-s}\\
\nonumber
&=\Pi^{(\ell)}_{s-3}\sum_{u: u\vartriangleright \ell}\sum_{
i,j:u\succcurlyeq i\succcurlyeq j\succ \ell}
A_{\ell j,-s+3}A_{ji,-s+2}A_{iu,-s+1}W_{u,-s}\\
&=\Pi_{s-3}^{(\ell)}\sum_{u:u\vartriangleright \ell}\sum_{h\in H'_3(\ell,u)}
\Big(\prod_{p=0}^2 A_{h(p)h(p+1),-s+3-p}\Big)W_{u,-s}\nonumber\\
\label{eta2}
&=\Pi_{s-3}^{(\ell)}\sum_{u:u\vartriangleright \ell}\pi'_{\ell u}(-s+3,-s+1)W_{u,-s}.
\end{align}
Similarly for any $0\le
 n\le s-1$ we obtain inductively 
\begin{equation}
\label{etan}
\eta_n=\Pi_{n}^{(\ell)}\sum_{u:u\vartriangleright \ell}\pi'_{\ell u}(-n,-s+1)W_{u,-s}. 
\end{equation}
The expression for $\zeta_n$ is similar but slightly different. Let us
 write it for $n=s-3$. From \eqref{xi2} we infer 
\begin{align}
\nonumber
\zeta_{s-3}
&=\Pi_{s-3}^{(\ell)}\Bigg(B_{\ell,-s+3}+\sum_{j:j\succ
 \ell}A_{\ell j,-s+3}B_{j,-s+2}+\sum_{j:j\succ \ell}A_{\ell j,-s+3}\sum_{i:i\succcurlyeq j}A_{ji,-s+2}B_{i,-s+1}\Bigg)\\
\label{zeta2}
&=\Pi_{s-3}^{(\ell)}\Bigg(B_{\ell,-s+3}+\sum_{j:j\vartriangleright
 \ell}\pi'_{\ell j}(-s+3,-s+3)B_{j,-s+2}+\sum_{i:i\vartriangleright
 \ell}\pi'_{\ell i}(-s+3,-s+2)B_{i,-s+1}\Bigg) 
\end{align} 
and the general formula for $\zeta_n$ is: 
\begin{equation}
\label{zetan}
\zeta_n=\Pi_{n}^{(\ell)}\Bigg(B_{\ell,-n}+\sum_{j:j\vartriangleright
\ell}\sum_{m=1}^{s-1-n}\pi'_{\ell j}(-n,-n-m+1)B_{j,-n-m}\Bigg).
\end{equation}
Therefore,
\[
Q_F(s)= \sum_{n=0}^{s-1} \xi_n = \underbrace{\sum_{n=0}^{s-1} \eta_n}_{Q_W(s)}+ \underbrace{\sum_{n=0}^{s-1} \zeta_n}_{Q_B(s)},
\]
where
\begin{align*}
Q_W(s)&=\sum_{n=0}^{s-1}\Pi_{n}^{(\ell)}\sum_{j:j\vartriangleright
 \ell}\pi'_{\ell j}(-n,-s+1)W_{j,-s}.
\end{align*}
Finally to obtain \eqref{qs1},
we use \eqref{piprim} and the identity:
\[
\Pi_{n}^{(\ell)}\prod_{p=0}^{s-1-n} A_{h'(p)h'(p+1),-n-p}=\prod_{p=0}^{s-1} A_{h(p)h(p+1),-p}
\]
for $n<s$, $h'\in H'_{s-n}(\ell,j)$ and $h\in H_{s}(\ell,j)$ such that $h(0)=h(1)=\ldots=h(n)=\ell$ and
 $h(n+p)=h'(p)$ for $p=0,\ldots,s-1-n$.
Each element $h\in H_{s}(\ell,j)$ is uniquely represented in this way.

\noindent
{\bf Step 2. Estimation of $Q_B(s)$.}\par
Recall that $\wt\alpha_\ell
 \le\alpha_j$ for $j \vartriangleright \ell$. Hence 
$\E A_{ij}^{\wt\alpha_\ell}<\infty$ and $\E
 B_i^{\wt\alpha_\ell}<\infty$ for any $i,j
 \trianglerighteq \ell$. We have
\begin{equation}\label{bmoment}
\E
 Q_B(s)^{\wt\alpha_\ell}=\E\Bigg\{\sum_{n=0}^{s-1}\Pi_{n}^{(\ell)}\Big(B_{\ell,-n} + 
\sum_{j:j\vartriangleright\ell}
\sum_{m=1}^{s-1-n}\pi'_{\ell j}(-n,-n-m+1)B_{j,-n-m}\Big)\Bigg\}^{\wt\alpha_\ell}.
\end{equation}
Due to Minkowski's inequality for $\wt\alpha_\ell>1$ and Jensen's
 inequality for $\wt\alpha_\ell \le 1$, \eqref{bmoment} is bounded by finite
 combinations of the quantities
\[
 \E\pi'_{\ell j}(-n,-n-m+1)^{\wt\alpha_\ell}<\infty\quad \textrm{and} \quad \E B_{j,-n-m}^{\wt\alpha_\ell}<\infty
\]
for $0\le n \le s-1$, $j:j\trianglerighteq \ell$ and
 $1\le m\le s-1-n$. 
Here we recall that $\pi'_{\ell j}(-n,-n-m+1)$ is a
 polynomial of $A_{ij},\,i,j \trianglerighteq \ell$ (see
 \eqref{piprim}).
Thus $\E Q_B(s)^{\wt \alpha_\ell}<\infty$.
\end{proof}

\begin{example}
 In order to grasp the intuition of decomposition \eqref{decomposition}, we consider the
 case $d=2$ and let 
\[
\bfA_t=
\left(\begin{array}{cc}
A_{11,t}&A_{12,t}\\
0&A_{22,t}
\end{array}\right).
\]
Then applying the recursions $W_{2,r}=A_{22,r}W_{2,r-1}+B_{2,r},\ -s
 \le r \le 1$ to the quantity $Q_F(s)$ of \eqref{qs:series}, we obtain 
\begin{align*}
Q_F(s)=&\sum_{n=0}^{s-1}\Pi^{(1)}_{n}D_{1,-n}\\
=&\sum_{n=0}^{s-1}\Pi^{(1)}_{n}(A_{12,-n}W_{2,-n-1}+B_{1,-n})\\
=&\sum_{n=0}^{s-1}\Pi^{(1)}_{n}A_{12,-n}\Pi^{(2)}_{-n-1,-s+1}W_{2,-s}\\
&+\sum_{n=0}^{s-1}\Pi^{(1)}_{n}\Big(B_{1,-n}+\sum_{m=1}^{s-1-n} 
A_{12,-n}\Pi^{(2)}_{-n-1,-n+1-m}B_{2,-n-m}\Big),
\end{align*}
where we recall that we use the convention $\Pi^{(\cdot)}_{-n-1,-n}=1$. 

Let us focus on the first sum in the last expression, which is equal
to $Q_W(s)$. 
Each term of this sum contains a product of $s+1$ factors of the form
 $A_{11,\cdot},\,A_{12,\cdot}$ or $A_{22,\cdot}$. Each product is
 completely characterized by the nondecreasing sequence
 $(h(i))_{i=0}^{s}$ of natural numbers with $h(0)=1$ and
 $h(s)=2$. If $h(i)=1$ for $i\le q$ and $h(i)=2$ for $i>q$, then there are $q$ factors of the form
 $A_{11,\cdot}$ in front of $A_{h(q)h(q+1),-q}=A_{12,-q}$, and $s-1-q$ factors of the form $A_{22,\cdot}$
 behind. All such sequences constitute $H_{s}(1
 ,2)$ of
 Definition \ref{sequences}. Thus we can write 
\[
Q_W(s)=\underbrace{\sum_{h\in H_{s+1}(1,2)}\prod_{p=0}^{s-1} A_{h(p)h(p+1),-p}}_{\pi_{1,2}(s)}W_{2,-s}.
\]
The second sum, which corresponds to $Q_B(s)$, has another sum of the
products in the $n$-th term. All terms in these secondary sums of $m$ are
starting with $A_{12,0}$ and then have a product of $A_{22,\cdot}$ until we reach $B_{\cdot,\cdot}$. Each
of products is again characterized by a nondecreasing sequence $(h(i))_{i=0}^m$
 such that $h(0)=1,\,h(m)=2$ and $h(1)\neq 1$, because there is just one
 such sequence for each $m$. Thus we can use $H_m'(1,2)$ of Definition
 \ref{sequences} and write 
\[
Q_B(s)=\sum_{n=0}^{s-1}\Pi^{(1)}_{n}\Big(B_{1,-n}+
\sum_{m=1}^{s-1-n}\underbrace{\sum_{h\in
H'_m(1,2)}
\prod_{p=0}^{m-1} A_{h(p)h(p+1),-n-p}}_{\pi'_{12}(-n,-n-m+1)}B_{2,-n-m}\Big).
\]
\end{example}

\subsection{The dominant term}\label{dominant}
Lemmas \ref{q-decomposition} and \ref{q^s_step} imply that 
the tail of $W_{\ell,0}$ is determined by   
\[
 Q_W(s) =\sum\limits_{j:j\vartriangleright \ell}\pi_{\ell j}(s)W_{j,-s}
\]
in \eqref{qs1}. In the subsequent Lemmas (Lemmas \ref{one_many} and \ref{one_one}) we
will apply the recurrence to $W_{j,-s},\ j\vartriangleleft j_0$ until
they reach $W_{j_0,t}$ for some $t<s$. Those $W_{j,-s}$ could survive as the dominant
terms, and terms $W_{j,-s},\ j\not\trianglelefteq j_0$ are proved to be negligible. In these steps the behaviors of coefficients for all $W_{j,-s}$ are inevitable.

The following property is crucial in subsequent steps,
particularly in the main proof.
\begin{lemma}
\label{u:finite}
Let $\alpha_i>\wt \alpha_i=\alpha_{j_0}$ for some
 $j_0\vartriangleright i$ and 
$u_i (s):=\E \pi_{i j_0}(s)^{\alpha_{j_0}}$. 
Then, the limit $u_i=\lim_{s\to\infty} u_i (s)$ exists, and it is finite and strictly positive.
\end{lemma}
\begin{proof}
First notice that $u_i(s)>0$ for $s$ large enough since $\pi_{ij}(s)$ is positive
 whenever $j\trianglerighteq i$ and $s\ge j-i$. We are going to prove that the sequence
 $u_i(s)$ is non-decreasing w.r.t. $s$. Observe that 
\begin{align*}
\pi_{ij_0}(s+1)&=\sum_{j:j_0\succcurlyeq j\trianglerighteq i}\pi_{ij}(s)A_{jj_0,-s}\\
& \ge\pi_{ij_0}(s)\cdot A_{j_0j_0,-s}.
\end{align*}
and therefore, by independence,
\[
u_i(s+1)=\E \pi_{ij_0}(s+1)^{\alpha_{j_0}} \ge\E
 \pi_{ij_0}(s)^{\alpha_{j_0}} \cdot\E A_{j_0j_0,-s}^{\alpha_{j_0}}
=\E \pi_{ij_0}(s)^{\alpha_{j_0}}=u_i(s).
\]
Since $u_i(s)$ is non-decreasing in $s$, if $(u_i(s))_{s\in\N}$ is
 bounded uniformly in $s$, the limit exists. For the upper bound on $u_i(s)$, notice that each product in $\pi_{ij}(s)$ (see \eqref{pi}) can be divided into two parts:
\[
\prod_{p=0}^{s-1}
 A_{h(p)h(p+1),-p}=\underbrace{\Big(\prod_{p=0}^{s-m-1}A_{h(p)h(p+1),-p}\Big)}_{\mathrm{first\ 
 part}}\cdot \underbrace{\Pi^{(j_0)}_{-s+m,-s+1}\vphantom{\prod_p^s}}_{\mathrm{second\
 part}}, 
\]
where $m$ denotes the length of the product of $A_{j_0j_0}$ terms. In the
 first part, all $h(\cdot)$ but the last one $h(s-m)=j_0$ are strictly smaller
 than $j_0$. 
Since $\E A_{j_0j_0}^{\alpha_{j_0}}=1$, clearly 
\[
\E\Big( \prod_{p=0}^{s-1} A_{h(p)h(p+1),-p}\Big)^{\alpha_{j_0}}=\E\Big(\prod_{p=0}^{s-m-1}A_{h(p)h(p+1),-p}\Big)^{\alpha_{j_0}}.
\]
Now let $H''_{n}(j,j_0)$ denote the set of all sequences $h\in
 H_{n}(j,j_0)$ which have only one $j_0$ at the end. 
Suppose first that $\alpha_{j_0}<1$. 
Then by Jensen's inequality we obtain
\begin{align}
\nonumber
\E \pi_{ij_0}(s)^{\alpha_{j_0}}
&=\E\Big(\sum_{h\in H_{s}(i,j_0)}\prod_{p=0}^{s-1} A_{h(p)h(p+1),-p}\Big)^{\alpha_{j_0}}\\
\nonumber
&\le\sum_{h\in H_{s}(i,j_0)}\prod_{p=0}^{s-1} 
\E A_{h(p)h(p+1),-p}^{\alpha_{j_0}}\\
\nonumber
&=\sum_{m=0}^{s-1}\sum_{h\in H''_{s-m}(i,j_0)}
\prod_{p=0}^{s-m-1} \E A_{h(p)h(p+1),-p}^{\alpha_{j_0}}\\
\label{h''1}
 &\le \sum_{m=0}^{s-1} d(s-m)^d\cdot M^{j_0-i}\cdot \rho^{s-m-(j_0-i)}\\
\label{h''2}
 &\le d\rho^{i-j_0} M^{j_0-i} \sum_{l=1}^\infty l^d\cdot \rho^l<\infty. 
 \end{align}
Here we put
\begin{equation}
\label{rhom}
\rho=\max\{\E A_{jj}^{\alpha_{j_0}}:j\vartriangleleft j_0\}<1;\qquad M=\max\{\E A_{uv}^{\alpha_{j_0}}:u\vartriangleleft j_0,v\trianglelefteq j_0\}<\infty.
\end{equation}

The term $d(s-m)^d$ in \eqref{h''1} is an upper bound on the number of elements of $H''_{s-m}(i,j)$. Another term $M^{j_0-1}$ bounds the contribution of non-diagonal elements in the product, since any sequence $h\in H_{s-m}(i,j_0)$ generates at most $j_0-i$ such elements. The last term $\rho^{s-m-(j_0-i)}$ in \eqref{h''1} is an estimate of contribution of the diagonal elements since there are at least $s-m-(j_0-i)$ of them.

We obtain \eqref{h''2} from \eqref{h''1} by substituting $l=s-m$ and extending the finite sum to the infinite series. Since the series converges and its sum does not depend on $s$, the expectation $\E\pi_{ij_0}(s)^{\alpha_{j_0}}$
is bounded from above uniformly in $s$.
 
Similarly, if $\alpha_{j_0}\ge 1$, then by Minkowski's inequality we obtain
\begin{align*}
\left(\E \pi_{ij_0}(s)^{\alpha_{j_0}} \right)^{1/\alpha_{j_0}}
&=\left(\E\Big(\sum_{h\in H_{s}(i,j_0)}\prod_{p=0}^{s-1} A_{h(p)h(p+1),-p}
\Big)^{\alpha_{j_0}}\right)^{1/\alpha_{j_0}}\\
&\le\sum_{h\in H_{s}(i,j_0)}\prod_{p=0}^{s-1} 
\left(\E A_{h(p)h(p+1),-p}^{\alpha_{j_0}}\right)^{1/\alpha_{j_0}}\\
&=\sum_{m=0}^{s-1}\sum_{h\in H''_{s-m}(i,j_0)}
\prod_{p=0}^{s-m-1}\Big(\E A_{h(p)h(p+1),-p}^{\alpha_{j_0}}\Big)^{1/\alpha_{j_0}},
\end{align*}
which is again bounded from above, uniformly in $s$, by the same argument.
\end{proof}

The number $u_i$ depends only on $i$, since the index $j_0=j_0(i)$ is
uniquely determined for each $i$. We are going to use $u_i$ as an upper
bound for $\E\pi_{ij}(s)^{\alpha_{j_0}}$ with any $j$ such that
$j_0 \trianglerighteq j\trianglerighteq i$. This is justified by the
following lemma.
\begin{lemma}
 \label{uij}
 For any $j$ such that $j_0 \vartriangleright
 j \trianglerighteq i $ there is $\hat s_j$ such that
 \begin{equation}
\label{tau}
\E \pi_{ij}(s)^{\alpha_{j_0}} <u_i \quad{\rm for}\ s>\hat s_j.
\end{equation}
\end{lemma}
\begin{proof}
The argument is similar to that in the proof of Lemma \ref{small_moment}. Indeed, one finds $\pi_{ij}(s)$ in $S_{i,n}$ of \eqref{eq:defSin} by setting $s=m$. We briefly recall the argument for $\alpha_{j_0}\le 1$. The case $\alpha_{j_0}>1$ is similar. Without loss of generality we set $s\ge 2d-1$. The number of sequences in $H_{s}(i,j)$ is less than $ds^d$. 
Taking $\rho$ and $M$ as in \eqref{rhom},
from \eqref{pi} and Jensen's inequality we infer
 \begin{align*}
  \E \pi_{ij}(s)^{\alpha_{j_0}}
  &=\E\left(\sum_{h\in H_{s}(i,j)}\prod_{p=0}^{s-1} A_{h(p)h(p+1),-p}\right)^{\alpha_{j_0}}\\
  &\le\sum_{h\in H_{s}(i,j)}\prod_{p=0}^{s-1} \E A_{h(p)h(p+1)}^{\alpha_{j_0}}\\
  &\le d M^d s^d \rho^{s-d}\ 
 \stackrel{s\to\infty}\longrightarrow 0, 
 \end{align*}
which implies \eqref{tau}. 
\end{proof}

We have already done all the preliminaries and we are ready for the goal of this section,
i.e. to establish the expression \eqref{main:expression}:
\[
 W_{\ell,0}= \pi_{\ell j_0}(s) W_{j_0,-s}+ R_{\ell j}(s)
\]
and prove the negligibility of the term $R_{\ell j}(s)$ in the tail when $s\to\infty$.

\begin{lemma}
\label{one_many}
Suppose that $\wt\alpha_\ell<\alpha_\ell$ and $j_0\vartriangleright
 \ell$ is the unique number with the property
 $\wt\alpha_\ell=\alpha_{j_0}$. 
Assume that
\[
\lim_{x\to\infty}x^{\alpha_{j_0}}\P(W_j>x)=C_j>0
\]
whenever $j_0\trianglerighteq j\vartriangleright \ell$. Then for any
 $\varepsilon>0$ there exists $s_\ell=s_\ell(\varepsilon)$ such that if $s>s_\ell$,
 $W_{\ell,0}$ has a representation 
\begin{align}
\label{d1} W_{\ell,0}=\pi_{\ell j_0}(s) W_{j_0,-s}+R_{\ell
 j_0}(s),
\end{align}
where $R_{\ell j_0}(s)$ satisfies 
\begin{align}
\label{d2}
&\varlimsup_{x\to\infty}x^{\alpha_{j_0}}\P(|R_{\ell j_0}(s)|>x)<\varepsilon.
\end{align}
\end{lemma}
The proof is given by induction and first we consider the case when indeed
$\ell\prec j_0$, not only $j_0\vartriangleright\ell$. More precisely, we prove the following lemma which serves as a basic tool in each inductive step.
\begin{lemma}
\label{one_one}
Suppose that $\alpha_\ell>\wt\alpha_\ell=\alpha_{j_0}$ for some $j_0\vartriangleright\ell$ 
and $\wt \alpha_j > \alpha_{j_0}$ for every $j\vartriangleright
 \ell,\,j\neq j_0$.
 Then for any $\varepsilon>0$
 there exists $s_\ell=s_\ell(\varepsilon)$ such that for any $s>s_\ell$, $W_{\ell,0}$ has a representation \eqref{d1} which satisfies
 \eqref{d2}. 
\end{lemma}
The condition in the Lemma says that $j_0$ is the unique coordinate
which has the heaviest tail among $j:j\vartriangleright \ell$, all the
other coordinates that determine $\ell$
do not depend on $j_0$ and they have lighter tails.

Notice that as long as we only represent $W_{\ell,0}$ by $\pi_{\ell
j_0}(s)W_{j_0,-s}$ plus some r.v., we need not take a large
$s$. Indeed, $s=0$ is enough to obtain $(\pi_{\ell j_0}(0)=A_{\ell
j_0,0})$ if $\ell\prec j_0$. 
Thus, the number $s_\ell$ is specific for \eqref{d2}. 

\begin{proof}
In view of \eqref{qs1} we may write
\begin{align}
\label{s-decomposition}
Q_W(s)&=\pi_{\ell j_0}(s)W_{j_0,-s}
+\underbrace{\sum_{\substack{j:j\vartriangleright \ell\\ j\neq j_0}}\pi_{\ell j}(s)W_{j,-s}}_{\hat{Q}_W(s)} \\
&= \pi_{\ell j_0}(s) W_{j_0,-s} +\hat{Q}_W(s), \nonumber
\end{align}
where $\E \hat{Q}_W(s)^{\alpha_{j_0}}<\infty$ by the two facts;\\
$(1)$ $\E \pi_{\ell j}(s)^{\alpha_{j_0}}<\infty$ for any $j$
(see the explanation below \eqref{pi}), \\
$(2)$ $\E W_j^{\alpha_{j_0}}<\infty$ for $j\vartriangleright \ell,\,j\neq j_0,$
 which is due to Lemma \ref{small_moment}. 

Thus by Lemma \ref{q-decomposition} we have 
\[
W_{\ell,0}=Q_W(s)+Q_B(s)+Q_T(s)=\pi_{\ell j_0}(s)W_{j_0,-s}+\hat{Q}_W(s)+Q_B(s)+Q_T(s),
\]
where $\hat{Q}_W(s)$ and $Q_B(s)$ have finite moment of order
 $\alpha_{j_0}$ and $Q_T(s)$ satisfies \eqref{q^s_estimate}. Now putting 
\begin{equation}
\label{R}
R_{\ell j_0}(s)=\hat{Q}_W(s)+Q_B(s)+Q_T(s),
\end{equation}
and setting $s_\ell(\varepsilon)=s_0(\ell,\varepsilon)$
 as in Lemma \ref{q^s_step}, 
we obtain the result. 
\end{proof}

\begin{proof}[Proof of Lemma \ref{one_many}]
 Here we may allow the existence of $j:j_0 \vartriangleright j \vartriangleright
 \ell$, so that there exist sequences $(j_i)_{0\le i\le n}$ such that 
 $j_0\vartriangleright j_1\vartriangleright
 \cdots\vartriangleright j_n=\ell$. Since these sequences are strictly
 decreasing, their lengths are at most $j_0-\ell+1$, i.e. possibly smaller
 than $j_0-\ell+1$. Let $n_0=n_0(\ell)$ denote the maximal length of sequence such
 that 
 $(j_i)_{1\le i\le n_0}$ satisfies $j_0\vartriangleright
 j_1\vartriangleright\cdots\vartriangleright j_{n_0}=\ell$.
 Then clearly $j_0\succ j_1\succ\cdots\succ j_{n_0}=\ell$.
In the same way we define $n_0(j)$ for any $j$ in the range $\ell\vartriangleleft j\trianglelefteq j_0$. We sometimes abbreviate to $n_0$ when we mean $n_0(\ell)$.
 We use induction with respect to this maximal number $n_0$
 to prove \eqref{d1} and \eqref{d2}. First we directly prove these two properties in the cases $n_0=0,1$ which serve as the induction basis. For $n_0>1$ the proof relies on the assumption that analogues of \eqref{d1} and \eqref{d2} hold for any $j$ with $n_0(j)<n_0(\ell)$.
The proof is divided into four steps.
 
{\bf Step 1. Scheme of the induction and the basis.}\par
Consider arbitrary coordinate $j_0$ satisfying $\alpha_{j_0}=\wt\alpha_{j_0}$. We prove that for any $\ell\trianglelefteq j_0$ satisfying $\wt\alpha_\ell=\alpha_{j_0}$ the properties \eqref{d1} and \eqref{d2} hold. The proof follows by induction with respect to $n_0(\ell)$. Namely, we assume that for any $j \trianglelefteq j_0$ satisfying $\wt\alpha_j=\alpha_{j_0}$ and $n_0(j)<n_0(\ell)$ and for any $\varepsilon>0$ there is $s_j=s_j(\varepsilon)$ such that for any $s>s_j$ we have
\begin{equation}
 W_{j,0} = \pi_{jj_0}(s)W_{j_0,-s}+R_{jj_0}(s), \label{ci1} 
 \end{equation}
 where $R_{jj_0}(s)$ satisfies 
\begin{align}
 \varlimsup_{x\to\infty} x^{\alpha_{j_0}}
 \P(|R_{jj_0}(s)|>x)<\varepsilon \label{ci2}.  
\end{align}
Then we prove \eqref{d1} and \eqref{d2}.

First we settle the induction basis. Here we consider the cases $n_0=0$ and $n_0=1$. The former case is equivalent to $\ell=j_0$. In that setting, Theorem \ref{main} was already proved in Lemma \ref{minimal_alpha}, nevertheless \eqref{d1} and \eqref{d2} need to be shown separately.
The iteration of \eqref{coordinate} yields then that 
\begin{align}
\label{wj0:recurrsion}
  W_{j_0,0}=A_{j_0j_0,0}W_{j_0,-1}+D_{j_0,0}=\ldots=\underbrace{\Pi^{(j_0)}_{s}}_{\pi_{j_0j_0}(s)}W_{j_0,-s}
  +\underbrace{\sum_{n=0}^{s-1}\Pi^{(j_0)}_{n}D_{j_0,-n}}_{R_{j_0j_0}(s)},
\end{align}
 where we recall that $\Pi^{(j_0)}_{0}=1$ and $\Pi^{(j_0)}_{1}=A_{j_0j_0,0}$. From the definition of $j_0$ and Lemma \ref{small_moment} it follows that $\E D_{j_0}^{\alpha_{j_0}}<\infty$.
 Since $R_{j_0j_0}(s)$ is constituted by a finite sum of ingredients which have the $\alpha_{j_0}$th moment, we conclude that $\E R_{j_0j_0}(s)^{\alpha_{j_0}}<\infty$ and \eqref{d2} holds for any $s$. Thus we may let $s_\ell(\varepsilon)=1$ for any $\varepsilon>0$.
The case $n_0=1$ is precisely the setting of Lemma \ref{one_one}, in which we have already proved \eqref{d1} and \eqref{d2}.

If $n_0>1$, then there is at least one coordinate $j$ satisfying $\ell\vartriangleleft j\trianglelefteq j_0$. For any such $j$ it follows that $n_0(j)<n_0(\ell)$, hence we are allowed to use the induction assumptions \eqref{ci1} and \eqref{ci2}. In the next step we prove that this range of $j$ is essential, while for any other $j\vartriangleright\ell$ the induction is not necessary.

{\bf Step 2. Decomposition of $Q_W(s)$ and estimation of the negligible term.}\par
The first term in \eqref{d1} comes from the part $Q_W(s)$ of $W_{\ell,0}$ in
 Lemma \ref{q-decomposition} and we further write 
  \begin{align}
\label{eq:decompQ}
  Q_W(s)=
 \underbrace{\sum_{j:j_0\trianglerighteq j\vartriangleright
 \ell}\pi_{\ell j}(s)W_{j,-s}}_{Q'_W(s)}+
 \underbrace{\sum_{j:j_0\not\trianglerighteq
 j\vartriangleright \ell}\pi_{\ell j}(s)W_{j,-s}}_{Q''_W(s)}.
 \end{align}
 Recall that the relation `$\not\trianglerighteq$' describes dependence
 between the components of the solution $\bfW_t$ after a finite number
 of iterations of \eqref{sre1}. Therefore the range of summation
 $j_0\not\trianglerighteq j\vartriangleright\ell$ in $Q''_W(s)$ means 
that $\ell\,(\neq j)$ depends on both $j$ and $j_0$ (by definition of $j_0$), but $j$ does not depend on $j_0$.
The relation $j\vartriangleright \ell$
 implies that $\wt\alpha_j\ge\wt\alpha_\ell$, while
 $j_0\not\trianglerighteq j$ yields that $\wt\alpha_j\neq
 \alpha_{j_0}$. Recalling that $\wt\alpha_\ell=\alpha_{j_0}$, we can say
 that in $Q'_W(s)$ we gather all coordinates $j:j\vartriangleright
 \ell$ such that $\wt\alpha_j=\alpha_{j_0}$ and $Q''_W(s)$ consists of
 $j:j\vartriangleright \ell$ such that $\wt\alpha_j>\alpha_{j_0}$. Hence by Lemma \ref{small_moment} each $W_j$ appearing in $Q''_W(s)$ has a tail of lower order than the tail of $W_{j_0}$.
 
 Notice that $\hat{Q}_W(s)$ defined in \eqref{s-decomposition} is the form that $Q''_W(s)$ takes under the assumptions of Lemma \ref{one_one}. In this special case we also have $Q'_W(s)=\pi_{\ell j_0}(s)W_{j_0,-s}$. We are going to study this expression in the more general setting of Lemma \ref{one_many}.

{\bf Step 3. The induction step: decomposition of $Q'_W(s)$.}\par
To investigate $Q_W'(s)$ in more detail we will
 introduce a shifted version of $R_{ij}(s)$. First
 recall that the r.v.'s $\pi_{ij}(t,s)$ and
 $\pi_{ij}(t-s+1)$ have the same distribution and $\pi_{ij}(t,s)$ can be
 understood as a result of applying $t$ times a shift to all time
 indices in $\pi_{ij}(t-s-1)$. In the same way we define $R_{ij}(t,s)$
 as a result of applying $t$ times a shift to all time indices in
 $R_{ij}(t-s-1)$. In particular we have $R_{ij}(0,-s+1)=R_{ij}(s)$. 
By the shift-invariance of the stationary solution,
\eqref{ci1} and \eqref{ci2} are equivalent to their time-shifted versions:
\begin{equation}
\label{ci1'}
W_{j,t}=\pi_{jj_0}(t,t-s+1)W_{j_0,t-s}+R_{jj_0}(t,t-s+1),\quad t\in\Z
\end{equation}
and
\begin{equation}
\label{ci2'}
\varlimsup_{x\to\infty} x^{\alpha_{j_0}}\P(|R_{jj_0}(t,t-s+1)|>x)<\varepsilon,\quad t\in\Z
\end{equation}
respectively.
Fix arbitrary numbers $s_1,s_2\in\N$.
Letting $t=-s_1$ in \eqref{ci1'}, we obtain
 
\begin{align}
 Q'_W(s_1) &= \sum_{j:j_0 \trianglerighteq j \vartriangleright \ell}
 \pi_{\ell j}(s_1)W_{j,-s_1} \nonumber \\
 &= \sum_{j:j_0 \trianglerighteq j \vartriangleright \ell} \pi_{\ell j}(s_1)
 \Big\{
 \pi_{jj_0}(-s_1,-s_1-s_2+1) W_{j_0,-s_1-s_2} + R_{jj_0}(-s_1,-s_1-s_2+1) 
\Big\}  \nonumber \\
 & = \sum_{j:j_0 \trianglerighteq  j \vartriangleright \ell}
\pi_{\ell j}(s_1)\pi_{jj_0}(-s_1,-s_1-s_2+1) W_{j_0,-s_1-s_2}\nonumber\\
&\qquad + \sum_{j:j_0 \trianglerighteq j \vartriangleright \ell}
 \pi_{\ell j}(s_1) R_{jj_0}(-s_1,-s_1-s_2+1) \nonumber \\
  & = \sum_{j:j_0 \trianglerighteq  j \trianglerighteq \ell}
\pi_{\ell j}(s_1)\pi_{jj_0}(-s_1,-s_1-s_2+1) W_{j_0,-s_1-s_2}
 \nonumber \\
 &\qquad\left.\begin{aligned}
 &-\pi_{\ell\ell}(s_1)\pi_{\ell j_0}(-s_1,-s_1-s_2+1)W_{j_0,-s_1-s_2}\nonumber \\
 &+ \sum_{j:j_0 \trianglerighteq j \vartriangleright \ell}
 \pi_{\ell j}(s_1) R_{jj_0}(-s_1,-s_1-s_2+1)\end{aligned}\right\rbrace =:Q^*_W(s_1,s_2) \nonumber \\
 & = \pi_{\ell j_0}(s_1+s_2)W_{j_0,-s_1-s_2}+Q^*_W(s_1,s_2), \label{eq:decompQ1}
\end{align}
where $\pi_{\ell j_0}(s_1+s_2)$ consists of all combinations
 $\pi_{\ell j}(s_1)\pi_{jj_0}(-s_1,-s_1-s_2+1)$ on $j:j_0 \trianglerighteq
 j \trianglerighteq \ell$. 
 This is clear when we recal that $\pi_{\ell j}(s_1), \pi_{jj_0}(-s_1,-s_1-s_2+1)$ and $\pi_{\ell j_0}(s_1+s_2)$ are appropriate entries of the matrices $\bfPi_{s_1},\bfPi_{-s_1,-s_1-s_2+1}$ and $\bfPi_{s_1+s_2}=\bfPi_{s_1}\bfPi_{-s_1,-s_1-s_2+1}$ respectively.
 
Now a combination of 
 \eqref{decomposition}, \eqref{eq:decompQ} and \eqref{eq:decompQ1} yields 
 \begin{align}
 W_{\ell,0} &= Q_W(s_1) + Q_B(s_1) +Q_T(s_1) \nonumber \\ 
 &= Q_W'(s_1)+Q_W''(s_1) + Q_B(s_1) +Q_T(s_1) \nonumber \\ 
\label{eq:remainer}
 &= \pi_{\ell j_0}(s_1+s_2)W_{j_0,-s_1-s_2} + Q''_W(s_1) + Q^*_W(s_1,s_2) + Q_B(s_1) +Q_T(s_1)\\  
 &= \pi_{\ell j_0}(s_1+s_2)W_{j_0,-s_1-s_2} +R_{\ell j_0}(s_1+s_2). \nonumber
 \end{align}

Our goal will be achieved by putting $s=s_1+s_2$ in \eqref{eq:remainer} and showing
 \eqref{d2} for $R_{\ell j_0}(s)$ with $s=s_1+s_2$.

{\bf Step 4. The induction step: estimation of the negligible terms.}\par
 To obtain \eqref{d2}, we evaluate the four ingredients of
$R_{\ell j_0}(s_1+s_2)$ of \eqref{eq:remainer},
 where the second hypothesis \eqref{ci2} of induction is used. 
 Three of them, $Q''_W(s_1),Q_B(s_1)$ and $Q_T(s_1)$, are nonnegative, hence, it is
 sufficient to establish an upper bound for each of them. The fourth term, $Q^*_W(s_1,s_2)$, may attain both positive and
 negative values, thus we are going to establish an upper bound for its absolute value.
 
 First, since the terms with $j:j_0 \not\trianglerighteq j \vartriangleright \ell$ in
 $Q_W''(s_1)$ of \eqref{eq:decompQ} satisfy the same condition as those of the sum $\hat{Q}_W(s)$ of the previous
 lemma, 
\begin{align}
\tag{\bf E.1}
 \E Q''_W(s_1)^{\alpha_{j_0}}<\infty
\end{align}
 holds. Secondly, 
\begin{align}
\tag{\bf E.2}
 \E Q_B(s_1)^{\alpha_{j_0}}<\infty
\end{align}
 holds in view of
Lemma \ref{q-decomposition}. 
 Moreover, by Lemma \ref{q^s_step} for any $\varepsilon'>0$
 there is $s_0=s_0(\ell,\varepsilon')$ such that for $s_1 > s_0$
 \begin{align}
\tag{\bf E.3}
\label{eq:remainer2}
  \varlimsup_{x\to \infty} x^{\alpha_{j_0}}\P(Q_T(s_1)>x) <\varepsilon'. 
 \end{align}

 For the evaluation of $Q^*_W(s_1,s_2)$, we will use
 \eqref{tau} and therefore we have to assume that $s_1>\max\{\hat s_j; j_0 \trianglerighteq j \vartriangleright \ell \}$
 where $\hat s_j$ are defined in \eqref{tau}. Furthermore, we fix arbitrary $\varepsilon''>0$ and assume that $s_2>\max\{s_j(\varepsilon''):j_0\trianglerighteq j\vartriangleright\ell\}$ where $s_j(\varepsilon'')$ are defined right before \eqref{ci1}.
Recall that $\pi_{\ell j}(s_1),\,j:j_0\vartriangleright j \vartriangleright \ell$ is
 independent of $R_{jj_0}(-s_1,-s_1-s_2+1)$ 
 and has finite moment of order $\alpha_{j_0}+\delta$ with some $\delta>0$. We use \eqref{tau}, \eqref{ci2'} and Lemma \ref{weak_breiman} to obtain 
\begin{align*}
 \varlimsup_{x\to\infty} x^{\alpha_{j_0}}\P(
\pi_{\ell j}(s_1)|R_{jj_0}(-s_1,-s_1-s_2+1)|>x
) \le \E \pi_{\ell j}(s_1)^{\alpha_{j_0}} \cdot \varepsilon'' \le
 u_\ell\cdot \varepsilon''.
\end{align*}
 The situation is different for $j=j_0$. We cannot use Lemma \ref{weak_breiman},
 because we do not know whether $\E \pi_{\ell j_0}(s_1)^{\alpha_{j_0}+\delta}<\infty$
 for some $\delta>0$. Indeed, it is possible that $\E A_{j_0j_0}^{\alpha_{j_0}+\delta}=\infty$
 for all $\delta>0$ and then clearly $\pi_{\ell j_0}(s_1)$ also does not have any moment
 of order greater than $\alpha_{j_0}$. However,
 $\E \pi_{\ell j_0}(s_1)^{\alpha_{j_0}}<\infty$ holds for any $s_1$ and this is enough
 to obtain the desired bound. Since the term $R_{j_0j_0}(s_2)$ is
 nonnegative (see \eqref{wj0:recurrsion}) and
 it was already proved to have finite moment of order $\alpha_{j_0}$, we obtain for any $s_1,s_2$
 \[
  \E\left\{\pi_{\ell j_0}(s_1)|R_{j_0j_0}(-s_1,-s_1-s_2+1)|\right\}^{\alpha_{j_0}}=\E \pi_{\ell j_0}(s_1)^{\alpha_{j_0}}\cdot\E R_{j_0j_0}(s_2)^{\alpha_{j_0}}<\infty
 \]
 and thus
 \[
   \varlimsup_{x\to \infty} x^{\alpha_{j_0}}\P(\pi_{\ell j_0}(s_1)R_{j_0j_0}(-s_1,-s_1-s_2+1)>x)=0.
 \]
Hence, setting $N=\# \{j:j_0 \trianglerighteq
  j\vartriangleright \ell \}$, we obtain that 
 \begin{align}
&  \varlimsup_{x\to\infty} x^{\alpha_{j_0}}
  \P(|Q_W^\ast(s_1,s_2)|>x)\\
&\le
  \varlimsup_{x\to\infty}x^{\alpha_{j_0}}\P\left(\sum_{j:j_0\trianglerighteq
  j\vartriangleright\ell}\pi_{\ell j}(s_1)|R_{jj_0}(-s_1,-s_1- s_2+1)|\right.\nonumber \\
& \hphantom{\le
  \varlimsup_{x\to\infty}x^{\alpha_{j_0}}\P\left(\sum_{j:j_0\trianglerighteq
  j\vartriangleright\ell}\right.}
  +\left. \pi_{\ell\ell}(s_1)\pi_{\ell j_0}(-s_1,-s_1-s_2+1)W_{j_0,-s_1-s_2}>x\vphantom{\sum_{j_0}}\right)\nonumber \\
&\le \varlimsup_{x\to\infty} x^{\alpha_{j_0}} \left(\sum_{j:j_0  \trianglerighteq
  j\vartriangleright \ell} \P\Bigg(\pi_{\ell j}(s_1)
  |R_{jj_0}(-s_1,-s_1-s_2+1)|>\frac{x}{N+1}\Bigg)\right. \nonumber \\
  &\hphantom{\le \varlimsup_{x\to\infty} x^{\alpha_{j_0}}}
  +\P\Bigg(\pi_{\ell\ell}(s_1)\pi_{\ell j_0}(-s_1,-s_1-s_2+1)W_{j_0,-s_1-s_2}>\frac{x}{N+1}\Bigg)\left.\vphantom{\sum_{j_0}}\right) \nonumber \\
&\le \sum_{j:j_0 \trianglerighteq
  j\vartriangleright \ell} \varlimsup_{x\to\infty} x^{\alpha_{j_0}}
  \P((N+1)\cdot \pi_{\ell j}(s_1) |R_{jj_0}(-s_1,-s_1-s_2+1)|>x) \nonumber \\
  &\quad +\varlimsup_{x\to\infty}x^{\alpha_{j_0}}\P((N+1)\cdot \pi_{\ell\ell}(s_1)\pi_{\ell j_0}(-s_1,-s_1-s_2+1)W_{j_0,-s_1-s_2}>x)\nonumber \\
& \le \sum_{j:j_0 \trianglerighteq
  j\vartriangleright \ell } (N+1)^{\alpha_{j_0}}u_\ell \cdot
  \varepsilon''+(N+1)^{\alpha_{j_0}}\rho^{s_1}
u_\ell\cdot C_{j_0}. 
  \label{Q1c}
 \end{align}
 For the last inequality we used Lemma \ref{weak_breiman} and the fact
 that $\pi_{\ell\ell}(s_1)=\Pi^{(\ell)}_{0,-s_1+1}$.
 Since $\rho=\E A_{\ell\ell}^{\alpha_{j_0}}<1$, there is $s'=s'(\varepsilon'')$ such that $\rho^{s_1}u_\ell\cdot C_{j_0}<(N+1)^{\alpha_{j_0}}u_\ell\cdot\varepsilon''$ for all $s_1>s'$. Then, recalling that the sum in the last expression contains at most $N-1$ nonzero terms, the final estimate is

 \begin{equation}
    \label{eq:remainer3}
\tag{\bf E.4}
    \varlimsup_{x\to\infty} x^{\alpha_{j_0}}
    \P(|Q^*_W(s_1,s_2)|>x)<(N+1)^{\alpha_{j_0}+1} u_\ell \cdot\varepsilon''.
 \end{equation}

Now we are going to evaluate $R_{\ell j_0}(s)$ of \eqref{d2}.
The desired estimate can be obtained only if $s$ is chosen properly.

For convenience we briefly recall the conditions on
 $s_1$ and $s_2$ that were necessary to obtain the estimates ({\bf
 E.1-E.4}). The inequalities ({\bf E.1}) and ({\bf E.2}) do not rely on
 any assumption on $s_1$ or $s_2$. The other relations are the
 following. Firstly, to obtain the inequality ({\bf E.3}) we
 need to assume that $s_1>s_0(\ell,\varepsilon')$.
 Secondly, the estimates $s_1>\max\{\hat{s}_j:j_0\trianglerighteq j\vartriangleright \ell\}$ and $s_2>\max\{s_j(\varepsilon''):j_0\trianglerighteq j\vartriangleright\ell\}$ are
 used to prove \eqref{Q1c}. Passing from \eqref{Q1c} to ({\bf E.4}) relies on the condition $s_1>s'$.

Now let 
\begin{align*}
 s&> s_0(\ell,\varepsilon') \vee \max\{\hat{s}_j:j_0 \trianglerighteq j
 \vartriangleright \ell\}\vee s'
+\max\{s_j(\varepsilon''):j_0\trianglerighteq j\vartriangleright\ell\}+1,
\end{align*}
where $\cdot\vee\cdot=\max\{\cdot,\cdot\}$. 
Then there are $s_1> s_0(\ell,\varepsilon')\vee \max\{\hat{s}_j:j_0 \trianglerighteq j \vartriangleright
 \ell\}\vee s'$ and
 $s_2> \max\{s_j(\varepsilon''):j_0\trianglerighteq j\vartriangleright\ell\}$
such that $s=s_1+s_2$.
Then
\[
 R_{\ell j_0}(s)=R_{\ell j_0}(s_1+s_2)=Q''_W(s_1) + Q^*_W(s_1,s_2) + Q_B(s_1) +Q_T(s_1).
\]
The numbers $s_1$ and $s_2$ were chosen in such a way that \eqref{eq:remainer2} and \eqref{eq:remainer3} hold.
The terms $Q''_W(s_1)$ and $Q_B(s_1)$ are negligible in the asymptotics. Therefore it follows that
\begin{align*}
\varlimsup_{x\to\infty} x^{\alpha_{j_0}}\P(R_{\ell j_0}(s)>x) &= \varlimsup_{x\to\infty} x^{\alpha_{j_0}} \P(Q^*_W(s_1,s_2)+Q_T(s_1)>x) \\
& \le \varlimsup_{x\to\infty} x^{\alpha_{j_0}}\P(Q^*_W(s_1,s_2)>x/2)+\varlimsup_{x\to\infty}x^{\alpha_{j_0}} \P(Q_T(s_1)>x/2) \\
&\le 2^{\alpha_{j_0}}\left((N+1)^{\alpha_{j_0}+1}u_\ell\cdot\varepsilon''+\varepsilon'\right). 
\end{align*}
Since $\varepsilon'$ and $\varepsilon''$ are arbitrary, we obtain
 \eqref{d2}. 
\end{proof}

\section{Applications}
\label{application}
Although there must be several applications, we focus
on the multivariate GARCH$(1,1)$ processes, which is our main
motivation. In particular, we consider the constant conditional
correlations model by \cite{bollerslev:1990} and \cite{jeantheau:1998},
which is the most fundamental multivariate GARCH process. Related results
are followings. 
The tail of multivariate GARCH$(p,q)$ has been investigated in 
\cite{fernandez:muriel:2009} but with the setting of Goldie's
condition. A bivariate GARCH$(1,1)$ series with a triangular setting has
been studied in \cite{matsui:mikosch:2016} and
\cite{damek:matsui:swiatkowski:2017}. Particularly in
\cite{damek:matsui:swiatkowski:2017}, detailed analysis was 
presented including exact tail behaviors of both price and volatility
processes. Since the detail of application is an analogue of the bivariate GARCH$(1,1)$, we only see how the
upper triangular SREs are constructed from multivariate GARCH processes. 

Let $\boldsymbol{\alpha}_0$ be a $d$-dimensional vector with positive
elements and let
$\boldsymbol{\alpha}$ and $\boldsymbol{\beta}$ be $d\times d$ upper
triangular
matrices such that non-zero elements are strictly positive. 
For a vector $\bfx=(x_1,\ldots,x_d)$, write 
$\bfx^\gamma=(x_1^\gamma,\ldots,x_d^\gamma)$ for $\gamma>0$. Then we say that $d$-dimensional series
$\bfX_t=(X_{1,t},\ldots,X_{d,t})',\,t\in \Z$ has GARCH$(1,1)$ structure
if it satisfies 
\[
 \bfX_t = \boldsymbol\Sigma_t \bfZ_t,
\]
where $\bfZ_t=(Z_{1,t},\ldots,Z_{d,t})'$ constitute an i.i.d. $d$-variable random vectors and the
matrix $\boldsymbol \Sigma_t$ is 
\[
 \boldsymbol \Sigma_t = diag(\sigma_{1,t},\ldots,\sigma_{d,t}).
\]
Moreover the system of volatility vector
$(\sigma_{1,t},\ldots,\sigma_{d,t})'$ is given by that of squared process
$\bfW_t=(\sigma_{1,t}^2, \ldots, \sigma_{d,t}^2)'$. Observe that
$\bfX_t=\boldsymbol \Sigma_t \bfZ_t=diag(\bfZ_t) \bfW_t^{1/2}$, so that
$\bfX_t^2 =diag(\bfZ_t^2) \bfW_t$. Then $\bfW_t$ is given by
the following auto-regressive model.  
\begin{align*}
 \bfW_t &= \boldsymbol\alpha_0 +\boldsymbol\alpha \bfX_{t-1}^2
 +\boldsymbol\beta \bfW_{t-1} \\
 &= \boldsymbol\alpha_0 +(\boldsymbol\alpha
 diag(\bfZ_t^2)+\boldsymbol\beta)\bfW_{t-1}. 
\end{align*}
Now putting $\bfB_t:=\boldsymbol\alpha_0$ and
$\bfA_t:=(\boldsymbol\alpha diag(\bfZ_t^2)+\boldsymbol\beta)$, we obtain
the SRE: 
$\bfW_t=\bfA_t \bfW_{t-1}+\bfB_t$ with $\bfA_t$ the upper triangular
with probability one. Each component of $\bfA_t$ is written as 
\[
 A_{ij,t}= \alpha_{ij} Z_{ij,t}^2+\beta_{ij},\quad i\le j\quad \text{and}\quad A_{ij,t}=0,\quad
 i>j\quad a.s. 
\]
Thus we could apply our main theorem to the squared volatility process
$\bfW_t$ and obtain the tail indices for $\bfW_t$. From this, we could
derive tail behavior of $\bfX_t$ as done in
\cite{damek:matsui:swiatkowski:2017}. 

Note that we have more
applications in GARCH type models. Indeed we are considering an applications
in BEKK-ARCH models, of which tail behavior has been investigated
with the diagonal setting (see \cite{pedersen:wintenberger:2017}). 
At there we should widen our results into the case where the
corresponding SRE takes values on whole real line. The extension is
possible if we assume certain restrictions and consider
positive and negative extremes separately. 
Since
the BEKK-ARCH model is another basic model 
in financial econometrics, the analysis with the
triangular setting would provide more flexible tools for empirical analysis.

\section{Conclusions and further comments}\label{conclusions}
\subsection{Constants}
In the bivariate case, we can obtain the exact form of constants for
regularly varying tails (see \cite{damek:matsui:swiatkowski:2017}).
The natural question is whether we can obtain the form of constants even
in the $d$-dimensional case. The answer is positive. 
We provide an example which illustrates the method of finding these constants when $d=4$. Let
\[
 \bfA=\left(\begin{array}{cccc}
       A_{11}&A_{12}&A_{13}&A_{14}\\
       0&A_{22}&A_{23}&A_{24}\\
       0&0&A_{33}&A_{34}\\
			 0&0&0&A_{44}
      \end{array}\right)
\]
and suppose that $\alpha_3<\alpha_4<\alpha_2<\alpha_1$. 

For coordinate $k=3,4$, we have 
\[
 C_k=\frac{\E\left[(A_{kk}W_k+B_k)^{\alpha_k}-(A_{kk}W_k)^{\alpha_k}\right]}{\alpha_k\E[A_{kk}^{\alpha_k}\log A_{kk}]},
\]
where $W_k$ is independent of $A_{kk}$ and $B_k$. This is the
Kesten-Goldie constant (see \cite[Theorem 4.1]{goldie:1991}). Indeed
since $W_4$ is a solution to the univariate SRE, we immediately obtain
the constant. Since the tail index of $W_3$ is equal to
$\wt\alpha_3=\alpha_3$, the constant follows by \eqref{k-tail} in Lemma
\ref{minimal_alpha}.
For the second coordinate we have an equation
$W_2\stackrel{d}=A_{22}W_2+A_{23}W_3+A_{24}W_4$. Since
$\wt\alpha_2=\alpha_3$ and $\wt\alpha_4>\alpha_3$, the term $A_{23}W_3$
dominates all others in the asymptotics. In view of \eqref{c_j} we obtain
\[
C_2=u_2\cdot C_3,
\]
where the quantity $u_2$ is given in Lemma \ref{u:finite}.

The situation seems more complicated for the first coordinate, 
because we have the condition $\wt\alpha_1=\wt\alpha_2=\alpha_3$ on the
SRE:  
$W_1\stackrel{d}=A_{11}W_1+A_{12}W_2+A_{13}W_3+A_{14}W_4$. This means
that the tail of $W_1$ comes from $W_2$ and $W_3$ both of which have
dominating tails, and we could not single out the dominant term. However, by Lemmas \ref{u:finite} and \ref{one_many} again we obtain a simple formula
\[
C_1=u_1\cdot C_3.
\]

We can write the general recursive formula for constants in any dimension:
\[
C_k=\begin{cases}
\frac{\E[(A_{kk}W_k+B_k)^{\alpha_k}-(A_{kk}W_k)^{\alpha_k}]}{\alpha_k
     \E[A_{kk}^{\alpha_k}\log A_{kk}]}\quad&{\rm if}\
     \wt\alpha_k=\alpha_k,\\ 
u_k\cdot C_{j_0}\quad&{\rm if}\ \wt\alpha_k=\alpha_{j_0}<\alpha_k.
\end{cases}
\]

Finally we notice that these $u_k$ have only closed form including
infinite sums. The exact values of $u_k$ seem to be impossible
and the only method to calculate them would be numerical approximations. The situation is
similar to the Kesten-Goldie constant (see
\cite{mikosch:samorodnitsky:tafakori}).

\subsection{Open questions}
In order to obtain the tail asymptotics of SRE such as \eqref{sre1}, 
the Kesten's theorem has been the key tool (see \cite{buraczewski:damek:mikosch:2016}). 
However, when the coefficients of SRE are upper triangular matrices as
in our case, the assumptions of the theorem are not satisfied, so that
we could not rely on the theorem. Fortunately in our setting, we can
obtain the exact tail asymptotic of each coordinate, which is $\P (W_k>x)
\sim C_k x^{-\wt \alpha_k}$. 
However, in general setting, one does not necessarily obtain such
asymptotic even in the upper triangular case. 

The example is given in \cite{damek:zienkiewicz:2017}, which we briefly
see. Let $\bfA$ be an upper triangular matrix with $A_{11}=A_{22}$
having the index $\alpha>0$. Then, depending on additional assumptions,
it can be either $\P(W_1>x)\sim Cx^{-\alpha}(\log x)^{\alpha/2}$, or
$\P(W_1>x)\sim C'x^{-\alpha}(\log x)^{\alpha}$ for some constant $C,\,C'>0$. 

There are many natural further questions to ask. What happens to the
solution if some indices $\alpha_i$ of different coordinates are equal? 
How we could find the tail asymptotics when the coefficient matrix is
neither in the Kesten's framework nor upper triangular?
Moreover, if $\bfA$ includes negative entries, could we derive the tail
asymptotics?
They are all open questions.

\appendix

\section{Negativity of top Lyapunov exponent}
\label{sec:lyapunov}
We provide the proof for negativity of $\gamma_\bfA$
in Section \ref{sec:pre_statio}. It can also be deduced from more general results of Straumann \cite{S} or Grencs\'er, Michaletzky and Orlovitz \cite{GGO} on the top Lyapunov exponent for block triangular matrices.
However, in our case there is a direct elementary approach based on equivalence
of norms and Gelfand's formula.

Since all matrix norms are equivalent we can use the norm
\[
||\bfA||_1 = \sum_{i=1}^d \sum_{j=1}^d |A_{ij}|,
\]
which is submultiplicative: $||\bfA \bfB||_1 \le
||\bfA||_1\,||\bfB||_1$. 
Then since $\bfA$ has non-negative entries, we have
$\E\|\bfA\|_1=\|\E \bfA\|_1$. Moreover,
for any $\varepsilon\in(0,1)$,
$\|\bfA\|_1^\varepsilon\le \|\bfA^\varepsilon\|_1$, where
$\bfA^\varepsilon$ denotes the matrix $\bfA$ with each entry raised to the power of  
$\varepsilon$.
We are going to apply these to the form of top Lyapunov exponent
$\gamma_{\bfA}$. 
For any $\varepsilon>0$, by Jensen's inequality we have 
\begin{align*}
 \gamma_{\bfA}
 &= \inf_{n\ge 1} (n\varepsilon)^{-1}\E\log ||\bfPi_{n}||_1^\varepsilon \\
 &\le  \inf_{n\ge 1} (n\varepsilon)^{-1}\log \E ||\bfPi_{n}||_1^\varepsilon. 
\end{align*}
Then from properties of $\|\cdot\|_1$ above, we infer that
\begin{align}
\label{ineq:stationarity}
\E\Vert\bfPi_{n}\Vert_1^\varepsilon\le\E\Vert\bfPi_{n}^\varepsilon\Vert_1=\Vert\E\bfPi_{n}^\varepsilon\Vert_1\le\Vert\E\bfPi_{n}^{(\varepsilon)}\Vert_1, 
\end{align}
where $\bfPi_{n}^{(\varepsilon)}=\bfA_0^\varepsilon\cdots\bfA_{-n+1}^\varepsilon$. The last inequality follows from the superadditivity of the function $f(x)=x^\varepsilon$. Since the matrices $\bfA_i$ are i.i.d., we have $\E\bfPi_{n}^{(\varepsilon)}=(\E\bfA^\varepsilon)^n$. Here we take the $n$-th power in terms of matrix multiplication. Hence
\[
 \gamma_{\bfA} \le  \inf_{n\ge 1} (n\varepsilon)^{-1}\log 
 \Vert(\E \bfA^{\varepsilon})^n\Vert_1 =\frac{1}{\varepsilon} \inf_{n\ge 1} \log \Vert(\E
 \bfA^{\varepsilon})^n\Vert_1^{1/n}. 
\]
From Gelfand's formula (e.g. \cite[(1.3.3)]{belitskii:lyubich:1988}),
for any matrix norm $\|\cdot\|$ we can write 
 \[
  \lim_{n\to \infty}\|(\E\bfA^\varepsilon)^n\|^{1/n} = \rho(\E\bfA^\varepsilon). 
 \]
Taking the norm $\|\cdot\|_1$, we obtain 
\[
 \gamma_{\bfA} \le \frac{1}{\varepsilon}\inf_{n\ge 1}\log\|(\E\bfA^\varepsilon)^n\|_1^{1/n}\le\frac{1}{\varepsilon}\lim_{n\to\infty}\log\|(\E\bfA^\varepsilon)^n\|_1^{1/n}= \frac{1}{\varepsilon} \log \rho(\E \bfA^{\varepsilon}).
\]
Hence it suffices to show that 
$\rho(\E\bfA^{\varepsilon})<1$. If $0<\varepsilon<\min\{\alpha_1,\ldots,\alpha_d\}$, then from condition (T-4)	
\begin{equation}
\label{epsilon_moment}
\E A_{ii}^\varepsilon<1,\quad i=1,\ldots,d. 
\end{equation}
Since the spectral radius $\rho(\E
\bfA^{\varepsilon})$ is the maximal eigenvalue of
$\E\bfA^{\varepsilon}$,
stationarity is implied by
\eqref{epsilon_moment}. 

\begin{remark}
 $(i)$ By the equivalence of matrix norms, the argument above works for any norm. In order to observe this, take a certain norm $\|\cdot\|$ and apply the inequality $\|\bfA\|\le c \|\bfA\|_1$. Then by \eqref{ineq:stationarity} we obtain
 $\E\|\bfPi_{n}\|^\varepsilon \le c^\varepsilon \|\E\bfPi_{n}^{(\varepsilon)}\|_1$. Since $\lim_{n\to\infty}c^{\varepsilon/n}=1$ for any constant $c>0$, the whole argument holds. \\
 $(ii)$ By the submultiplicativity of $\|\cdot\|_1$, it is immediate to see that
 \[
  \gamma_{\bfA} \le \frac{1}{\varepsilon} \log \|\E
 \bfA^\varepsilon\|_1.  
 \]
 However, we do not have any control on the norm $\|\E\bfA^\varepsilon\|_1$, in particular it can be greater than $1$ for any $\varepsilon$. It is essential in our situation that $\rho(\E\bfA^\varepsilon) \le \|\E\bfA^\varepsilon\|_1$ and involving Gelfand's formula is necessary to obtain the desired bound. 
\end{remark}

\section{Version of Breiman's lemma}
We provide a slightly modified version of the classical Breiman's lemma
(e.g. \cite[Lemma B.5.1]{buraczewski:damek:mikosch:2016}), 
since it is needed in the proof for \eqref{d2} of Lemma \ref{one_many}. 
In the Breiman's lemma, we usually assume regular variation for the
dominant r.v.'s of the two, which we could not apply in our situation. 
Instead, we require only an upper estimate of the tail. 
The price of weakening assumptions is also a weaker result: 
on behalf of the exact asymptotics of a product, we obtain just an
estimate from above. 
The generalization is rather standard but we include it for completeness.
\begin{lemma}
\label{weak_breiman}
Assume that $X$ and $Y$ are independent r.v.'s and for some $\alpha>0$ the following conditions hold:
\begin{align}
\label{breiman_limsup}
&\varlimsup_{x\to\infty}x^\alpha\P(Y>x)<M\quad{\rm for\ a\ constant}\ M>0 ;\\
\label{breiman_moment}
&\E X^{\alpha+\varepsilon}<\infty\quad{\rm for\ some}\ \varepsilon>0.
\end{align}
Then
\begin{equation*}
\varlimsup_{x\to\infty}x^\alpha\P(XY>x)\le M\cdot\E X^\alpha.
\end{equation*}
\end{lemma}
\begin{proof}
The idea is the same as that in the original proof of Breiman's lemma, see e.g. \cite[Lemma B.5.1]{buraczewski:damek:mikosch:2016}. Let $P_X$ denote the law of $X$. Then for any fixed $m>0$ we can write
\begin{align*}
\P(XY>x)&=\int_{(0,\infty)}\P(Y>x/z)P_X(dz)\\
&=\Big(\int_{(0,m]}+\int_{(m,\infty),\, x/z > x_0} + \int_{(m,\infty),\, x/z \le x_0}
\Big)\,
\P(Y>x/z)P_X(dz).
\end{align*}
By \eqref{breiman_limsup}, there is $x_0$ such that $x^\alpha\P(Y>x)\le
 M$ uniformly in $x\ge x_0$.
For the first integral since $(x/z)\ge x_0$ for $x\ge mx_0$ and
 $z\in(0,m]$, by Fatou's lemma 
\begin{align}
\label{eq:bre1}
 \varlimsup_{x \to \infty} x^\alpha \int_{(0,m]} \P(Y>x/z)P_X(dz) &\le 
\int_{(0,m]}z^\alpha\varlimsup_{x\to\infty}(x/z)^\alpha\P(Y>x/z)P_X(dz) \\
 & \le M\cdot \int_{(0,m]} z^\alpha P_X(dz)
 \stackrel{m\to\infty}\rightarrow M\cdot\E X^\alpha. \nonumber
\end{align}
Since $x/z >x_0$, the same argument as above is applicable to the second
 integral: 
\begin{align}
\label{eq:bre2}
 \varlimsup_{x\to\infty}x^\alpha \int_{(m,\infty),\, x/z > x_0}\P(Y>x/z)
 P_X(dz) &\le
 \int_{z>m}z^\alpha\varlimsup_{x\to\infty}(x/z)^\alpha\P(Y>x/z)P_X(dz) \\
 &\le  M\cdot\int_{z>m}z^\alpha P_X(dz)\stackrel{m\to\infty}\rightarrow 0.\nonumber 
\end{align}
The assumption \eqref{breiman_moment} allows us to use Markov's
 inequality to estimate the last integral: 
\begin{align}
\label{eq:bre3}
 \int_{(m,\infty),x/z\le x_0}\P(Y>x/z)\P_X(dz)\le\int_{x/z\le x_0}P_X(dz)=\P(X>x/x_0)\le(x/x_0)^{-(\alpha+\varepsilon)}\E X^{\alpha+\varepsilon}
\end{align}
and therefore, the last integral is negligible as $x\to\infty$ regardless
 of $m$. Now in view of \eqref{eq:bre1}-\eqref{eq:bre3}, letting $m$ to
 infinity, we obtain the result. 
\end{proof}

\section{Acknowledgments}
The authors are grateful to Ewa Damek and Dariusz Buraczewski for valuable discussion on the subject of the paper.


{\small
\begin{thebibliography}{99}
\baselineskip12pt


\bibitem{alsmeyer:mentmeier2012}
{\sc Alsmeyer, G. and Mentemeier, S.}\ (2012)
Tail behavior of stationary solutions of random difference equations: the case of regular matrices. 
 {\em J. Difference Equ. Appl.} {\bf 18}, 1305--1332.


\bibitem{belitskii:lyubich:1988}
{\sc Belitskii, G. R. and Lyubich, Yu. I.}\ (1988) 
{\em Matrix norms and their applications}. Birkh\"auser, Basel.

\bibitem{bingham:goldie:teugels:1987}
{\sc Bingham, N.H., Goldie, C.M.\ and Teugels, J.L.}\ (1987)
{\em Regular Variation}. Cambridge University Press, Cambridge (UK).

\bibitem{bollerslev:1990}
{\sc Bollerslev, T.}\ (1990)
Modelling the coherence in short-run nominal exchange rates: a
multivariate 
generalised ARCH model. {\em Review of Economics and Statistics} {\bf
  72}, 498--505.

\bibitem{bougerol:picard:1992a}
{\sc Bougerol, P. and Picard, N. }\ (1992)
Strict stationarity of generalized autoregressive processes.
{\em Ann. Probab.} {\bf 20}, 1714--1730.

\bibitem{brandt:1986}
{\sc  Brandt, A.} (1986) The stochastic equation
$Y_{n+1}=A_n\,Y_n+B_n$ with stationary coefficients.
{\em Adv. in Appl. Probab.} {\bf 18}, 211--220.

\bibitem{buraczewski:damek:2010}
{\sc Buraczewski, D. and Damek, E.}\ (2010)
 Regular behavior at infinity of stationary measures of stochastic recursion on NA groups.
{\em Colloq. Math.} {\bf 118}, 499--523. 

\bibitem{buraczewski:damek:guivarch2009}
{\sc Buraczewski, D., Damek, E., Guivarc'h, Y., Hulanicki, A. and Urban, R.}\ (2009) 
Tail-homogeneity of stationary measures for some multidimensional stochastic recursions.
{\em Probab. Theory Related Fields} {\bf 145}, 385--420.

\bibitem{buraczewski:damek:mikosch:2016}
{\sc Buraczewski, D., Damek, E. and Mikosch, T.}\ (2016)
{\em Stochastic Models with Power-Law Tails. The Equation $X=AX+B$}.
Springer Int. Pub., Switzerland.

\bibitem{damek:matsui:swiatkowski:2017}
{\sc Damek, E., Matsui, M. and \'{S}wi\k{a}tkowski, W.}\ (2016)
Componentwise different tail solutions for bivariate stochastic
	recurrence equations with application to GARCH(1,1) processes. 
{\it Colloq. Math.} {\bf 155}, 227--254. 

\bibitem{damek:zienkiewicz:2017}
{\sc Damek, E. and Zienkiewicz, J.}\ (2018)
Affine stochastic equation with triangular matrices. 
 {\em J. Difference Equ. Appl.} {\bf 24}, 520--542.


\bibitem{embrechts:kluppelberg:mikosch:1997}
{\sc Embrechts, P., Kl\"uppelberg, C. and Mikosch, T.}\ (1997)
{\em Modelling Extremal Events for Insurance and Finance}.
Springer,  Berlin.

\bibitem{fernandez:muriel:2009}
{\sc Fern\'andez, B. and Muriel, N.}\ (2009)
Regular variation and related results for the multivariate GARCH(p,q)
model with constant conditional correlations.
{\em J. Multivariate Anal.} {\bf 100}, 1538--1550.



\bibitem{GGO}
{\sc Gerencs\'er, L., Michaletzky, G. and Orlovits, Z.}\ (2008)
Stability of block-triangular stationary random matrices. 
{\em Systems Cintrol Lett.} {\bf 57}, 620--625.

\bibitem{goldie:1991}
{\sc Goldie, C.M.}\ (1991)
Implicit renewal theory and tails of solutions of random equations.
{\em Ann.\ Appl.\ Probab.} {\bf 1}, 126--166.

\bibitem{guivarch:lepage:2015}
{\sc Guivarc'h, Y. and Le Page, \'E.}\ (2016)
Spectral gap properties for linear random walks and Pareto's asymptotics for affine stochastic recursions. 
{\em Ann. Inst. H. Poincar\'e Probab. Statist.} {\bf 52}, 503--574.

\bibitem{horvath:boril:2016}
{\sc Horv\'ath, R. and Boril \u{S}.}\ (2016)  
GARCH models, tail indexes and error distributions: An empirical
	investigation.
{\em The North American J. Economics and Finance} {\bf 37}, 1--15.

\bibitem{jeantheau:1998}
{\sc Jeantheau, T.}\ (1998)
Strong consistency of estimators for multivariate ARCH models. 
{\em Econometric Theory} {\bf 14}, 70--86. 

\bibitem{kesten:1973}
{\sc Kesten, H.}\ (1973)
Random difference equations and renewal theory for
products of random matrices.
{\em Acta Math.} {\bf 131}, 207--248.

\bibitem{kluppelberg:pergamentchikov:2004}
{\sc Kl\"uppelberg, C. and Pergamenchtchikov, S.}\ (2004)
The tail of the stationary distribution of a random coefficient AR($q$) model.
{\em Ann. Appl. Probab.} {\bf 14}, 971--1005

\bibitem{krengel}
{\sc Krengel, U.}\ (1985)
{\em Ergodic theorems}. De Gruyter, Berlin-New York

\bibitem{mcneil:fre:embrechts:2015}
{\sc McNeil, A.J., Frey, R. and Embrechts, P.}\ (2015) 
{\em Quantitative Risk Management: Concepts, Techniques and Tools
	(Princeton Series in Finance)}. Princeton Univ. Pr., Princeton.

\bibitem{matsui:mikosch:2016}
{\sc Matsui, M. and Mikosch, T.}\ (2016)
The extremogram and the cross-extremogram for a bivariate
        GARCH(1,1) process. 
{\em Adv. in Appl. Probab.} {\bf 48A}, 217--233.

\bibitem{mikosch:samorodnitsky:tafakori}
{\sc Mikosch, T., Samorodnitsky, G. and Tafakori, L.}\ (2013) 
Fractional moments of solutions to stochastic recurrence equations.
{\em J. Appl. Probab.} {\bf 50}, 969--982.

\bibitem{pedersen:wintenberger:2017} 
{\sc Pedersen, R.S. and Wintenberger, O.}\ (2018) 
On the tail behavior of a class of multivariate conditionally
	heteroskedastic processes. 
{\em Extremes} {\bf 21}, 261--284. 

\bibitem{resnick:1987}
{\sc Resnick, S.I.}\ (1987)
{\em Extreme Values, Regular Variation, and Point Processes}.
Sprin\-ger, New York.

\bibitem{resnick:2007}
{\sc Resnick, S.I.}\ (2007)
{\em Heavy-Tail Phenomena: Probabilistic and Statistical Modeling}.
Springer, New York.


\bibitem{S}
{\sc Straumann, D.}\ (2005)
{\em Estimation in conditionally heteroscedastic time series models}.
Lecture Notes in Statistics 181, Springer-Verlag, Berlin.

\bibitem{sun:chen:2014}
{\sc Sun, P. and Zhou, C.}\ (2014)
Diagnosing the distribution of GARCH innovations. 
{\em J. Empirical Finance} {\bf 29} 287--303.


\end{thebibliography}}
\end{document}